\documentclass[10pt]{article}

\usepackage{authblk, geometry, enumerate, amsmath, amssymb, amsthm, amscd, hyperref, graphicx, color, enumerate, mathrsfs, caption, subcaption, float, verbatim, tikz, calc, multicol}

\usetikzlibrary{decorations.markings}
\usetikzlibrary{arrows, shapes}
\usetikzlibrary{arrows.meta, shapes.misc}

\tikzstyle{vertex}=[circle,draw,inner sep=0pt, minimum size=6pt]

\newtheorem{thm}{Theorem}[section]
\newtheorem{cor}[thm]{Corollary}

\newtheorem{lem}[thm]{Lemma}

\newtheorem{prop}[thm]{Proposition}

\newtheorem{innercustomthm}{Theorem}
\newenvironment{customthm}[1]
  {\renewcommand\theinnercustomthm{#1}\innercustomthm}
  {\endinnercustomthm}

\theoremstyle{definition}\newtheorem{Rem}[thm]{Remark}
\theoremstyle{definition}\newtheorem{Def}[thm]{Definition}
\theoremstyle{definition}

\newcommand{\N}{\mathbb{N}}

\renewcommand{\le}{\leqslant}
\renewcommand{\ge}{\geqslant}

\newcommand{\mcT}{\mathcal{T}}
\newcommand{\mcU}{\mathcal{U}}

\newcommand{\lastFunc}{\mathfrak{r}}
\newcommand{\oom}{\overline{\lastFunc}}

\begin{document}

\title{On the number of reachable pairs in a digraph}
\markright{Reachable pairs in digraphs}

\author{Eric Swartz\footnote{Department of Mathematics, William \& Mary, Williamsburg, VA, USA.\newline E-mail: easwartz@wm.edu}\\
Nicholas J.\ Werner\footnote{Department of Mathematics, Computer and Information Science, State University of New York College at Old Westbury, Old Westbury, NY, USA.\newline E-mail: wernern@oldwestbury.edu}
}

\maketitle 

\thispagestyle{empty}

\begin{abstract}
A pair $(u, v)$ of (not necessarily distinct) vertices in a directed graph $D$ is called a reachable pair if there exists a directed path from $u$ to $v$. We define the weight of $D$ to be the number of reachable pairs of $D$, which equals the sum of the number of vertices in $D$ and the number of directed edges in the transitive closure of $D$. In this paper, we study the set $W(n)$ of possible weights of directed graphs on $n$ labeled vertices. We prove that $W(n)$ can be determined recursively and describe the integers in the set. Moreover, if $b(n) \ge n$ is the least integer for which there is no digraph on $n$ vertices with exactly $b(n)+1$ reachable pairs, we determine $b(n)$ exactly through a simple recursive formula and find an explicit function $g(n)$ such that $|b(n) - g(n)| < 2n$ for all $n \ge 3$. Using these results, we are able to approximate $|W(n)|$---which is quadratic in $n$---with an explicit function that is within $30n$ of $|W(n)|$ for all $n \ge 3$, thus answering a question of Rao.  Since the weight of a directed graph on $n$ vertices corresponds to the number of elements in a preorder on an $n$ element set and the number of containments among the minimal open sets of a topology on an $n$ point space, our theorems are applicable to preorders and topologies.  
\end{abstract} 

\section{Introduction}

Let $D$ be a directed graph (digraph) with vertex set $V$ and directed edge set $E$.  We will assume that digraphs do not contain loops or parallel edges (although these assumptions are irrelevant for this problem), but we do allow for pairs of oppositely oriented directed edges, e.g., given vertices $x$ and $y$, we allow both $(x,y)$ and $(y,x)$ to be in the directed edge set.  A \textit{reachable pair} is an ordered pair $(u, v)$ of vertices such that, for some nonnegative integer $k$, there exists a sequence of vertices $x_0 = u, x_1, \dots, x_k = v$ with the property that there is a directed edge from $x_i$ to $x_{i+1}$ for each $i$, $0 \le i \le k-1$.  In other words, $(u, v)$ is a reachable pair if there exists a directed path from $u$ to $v$.  We allow $k = 0$ for this, and so $(u, u)$ is considered a reachable pair for each vertex $u$. 

The determination of reachability in digraphs has been the object of considerable study in algorithmic design.  It is readily seen that the problem of determining the number of reachable pairs in a given digraph $D$ is equivalent to finding the size of the \textit{transitive closure} $\overline{D}$ of $D$: $\overline{D}$ has the same vertex set as $D$, but $\overline{D}$ contains the directed edge $(u, v)$ if and only if there is a directed path from $u$ to $v$ in $D$.  This problem has clear implications to communication within a network and is also important for many database problems, such as database query optimization.  For a discussion of this problem, see \cite[Section 15.5]{Skiena}.

\begin{figure}[H]
\centering 
\begin{tikzpicture}[x = 0.5cm, y = 0.5cm, style=thick, >= triangle 60]
\node[vertex, fill = black!20] (n1) at (-7,2)[]{};
\node[vertex, fill = black!20] (n2) at (-7,0)[]{};
\node[vertex, fill = black!20] (n3) at (-5,2)[]{};
\node[vertex, fill = black!20] (n4) at (-5,0)[]{};
\node[vertex, fill = black!20] (n5) at (-3,1)[]{};

\node[vertex, fill = black!20] (t1) at (3,2)[]{};
\node[vertex, fill = black!20] (t2) at (3,0)[]{};
\node[vertex, fill = black!20] (t3) at (5,2)[]{};
\node[vertex, fill = black!20] (t4) at (5,0)[]{};
\node[vertex, fill = black!20] (t5) at (7,1)[]{};
           
\path[->]
(n1) edge (n2)
(n1) edge (n3)
(n2) edge (n4)
(n3) edge (n5)
(n4) edge (n3);

\path[->]
(t1) edge (t2)
(t1) edge (t3)
(t2) edge (t4)
(t3) edge (t5)
(t4) edge (t3)
(t1) edge (t4)
(t1) edge[bend left = 80] (t5)
(t2) edge (t3)
(t2) edge[bend right = 80] (t5)
(t4) edge  (t5);
\end{tikzpicture}
\caption{A digraph $D$ (left) with its transitive closure $\overline{D}$ (right)}
\label{Fig1}
\end{figure}

A natural question along these lines is the following: given a positive integer $n$, what are the possible numbers of reachable pairs in a digraph on $n$ vertices?  Obviously, by our definition, there must always be at least $n$ reachable pairs, since $(u, u)$ is a reachable pair for each vertex $u$; and there are at most $n^2$ total pairs, so the number of reachable pairs is at most $n^2$, which occurs in a complete directed graph.  

\begin{Def}\label{Digraph weight def}
We define the \textit{weight} of a digraph $D$, denoted $w(D)$, to be the number of reachable pairs in $D$, which is equal to the sum of the number of vertices of $D$ and the number of directed edges in the transitive closure of $D$. For each $n \in \N$, we define
\begin{equation*}
W(n) = \{k \in \N : \text{ there exists a digraph $D$ on $n$ vertices of weight $k$}\}
\end{equation*}
to be the set of all possible weights of an $n$ vertex digraph. We call $W(n)$ a \textit{weight set}.
\end{Def}

For example, the digraphs shown in Figure \ref{Fig1} each have exactly $15$ reachable pairs, and so $w(D) = w(\overline{D}) = 15$. 

Our terminology here is inspired by the standard vocabulary used for weighted graphs. If each directed edge of $\overline{D}$ is assigned a weight of 1, then the weight of $D$ is simply $n$ plus the total weight of the directed graph $\overline{D}$. Usually, we will assume that $D$ itself is transitive, i.e., that $D = \overline{D}$.  

\begin{table}
\noindent\makebox[\textwidth]{%
{ \begin{tabular}{c | c | c}
$n$ & $W(n)$ & $[n, n^2] \setminus W(n)$\\
\hline
1 & 1 & $\varnothing$\\[.1cm] 
2 & $[2,4]$ & $\varnothing$\\[.1cm]
3 & $[3, 7]$, 9 & 8\\[.1cm]
4 & $[4, 13]$, 16 & 14, 15\\[.1cm]
5 & $[5, 19]$, 21, 25 & 20, $[22, 24]$\\[.1cm]
6 & $[6, 28]$, 31, 36 & 29, 30, $[32, 35]$\\[.1cm]
7 & $[7, 35]$, $[37, 39]$, 43, 49 & 36, $[40, 42]$, $[44, 48]$\\[.1cm]
8 & $[8, 52]$, 57, 64 & $[53, 56]$, $[58, 63]$\\[.1cm]
9 & $[9, 61]$, 63, $[65, 67]$, 73, 81 & 62, 64, $[68, 72]$, $[74, 80]$\\[.1cm]
10 & $[10, 77]$, 79, $[82, 84]$, 91, 100 & 78, 80, 81, $[85, 90]$, $[92, 99]$\\[.3cm] 
11 
& 
\begin{tabular}{@{}c@{}}
$[11, 95]$, 97, $[101, 103]$, \\ 111, 121
\end{tabular}
& 
\begin{tabular}{@{}c@{}}
96, $[98, 100]$, $[104, 110]$, \\
$[112, 120]$
\end{tabular}\\[0.5cm] 

12 & 
\begin{tabular}{@{}c@{}}
$[12, 109]$, $[111, 115]$, 117, \\$[122, 124]$, 133, 144 
\end{tabular}
& 
\begin{tabular}{@{}c@{}}
110, 116, $[118, 121]$, \\
$[125, 132]$, $[134, 143]$
\end{tabular}
\end{tabular}}}
\caption{Possible numbers of reachable pairs in a digraph with $n$ vertices.}
\label{W table}
\end{table}

Weight sets for directed graphs can be easily determined for small values of $n$. For example, we have $W(2) = \{2, 3, 4\}$, since, if $V = \{u, v\}$, we may choose $E$ to be $\varnothing$, $\{(u, v)\}$, or $\{(u, v), (v, u)\}$. So, when $n=2$, all values between $n$ and $n^2$ occur as possible weights. On the other hand, when $n = 3$, it is impossible for there to be exactly eight reachable pairs: if $V = \{u, v, x\}$ and $(u, v)$ is the unique pair that is not reachable, then both $(u, x)$ and $(x, v)$ are reachable pairs, meaning $(u, v)$ is also reachable by transitivity, a contradiction.  It is not difficult to see that $W(3) = [3,7] \cup \{9\}$, where $[a,b]$ denotes the set of  integers $k$ such that $a \le k \le b$. Table \ref{W table} lists the values in $W(n)$ for $n \le 12$, which may be determined through brute force calculations. 

From Table \ref{W table}, one can see that $W(n)$ becomes more fragmented as $n$ increases. Nevertheless, there are intriguing patterns in this data. For instance, $W(n)$ never includes integers in the range $[n^2-n+2, n^2-1]$, and for $n \ge 5$, $W(n)$ does not meet $[n^2-2n+5, n^2-n]$  (see \cite[Corollary 2]{Rao}). Moreover, $W(n)$ always begins with a single interval that contains the majority of the elements of the set, which motivates the following definition.

\begin{Def}\label{def:b} 
For each $n \in \N$, we define $b(n)$ to be the least integer such that $b(n) \ge n$ and there does not exist a digraph on $n$ vertices with exactly $b(n) + 1$ reachable pairs.  Equivalently, $b(n)$ is the largest positive integer such that $[n, b(n)] \subseteq W(n)$.
\end{Def} 

The set $W(n)$ and the integer $b(n)$ were studied previously in \cite{Rao}, although there the reachable pairs did not include pairs of the form $(u, u)$.  Hence, the set $S(n)$ studied in \cite{Rao} is related to $W(n)$ by
\[S(n) = \{k - n : k \in W(n)\},\]
and the function $f(n)$ studied in \cite{Rao} is related to $b(n)$ by $f(n) = b(n) - n$.  Indeed, \cite[Theorem 6]{Rao}, when translated into our notation, gives a lower bound of
\[ b(n) \ge n^2 - n \cdot \lfloor n^{0.57} \rfloor + \lfloor n^{0.57} \rfloor.\]
However, this bound is not asymptotically tight, as it is noted in \cite{Rao} that the lower bound holds for large enough $n$ if the exponent $0.57$ is replaced by $0.53$. Techniques for determining the set $S(n)$ for $n \le 208$ are given in \cite{Rao}, although an efficient method for calculating this set in general or even estimating its size is left as an open problem.

The purpose of this paper is to study the set $W(n)$ and the function $b(n)$.  First, we establish methods to determine $W(n)$ exactly. It follows from Rao's result \cite[Theorem 6]{Rao}, that $b(n) \ge (3/4)n^2$ for all $n \ge 8$, and hence $[n, \lceil (3/4)n^2\rceil] \subseteq W(n)$. We then prove that the larger values in $W(n)$ can realized by transitive digraphs with a particularly nice form. A \textit{mother vertex} of $D$ is a vertex $u$ such that, for all vertices $v \neq u$, there is a directed edge from $u$ to $v$. 

\begin{thm}\label{thm:rearrange}
Let $D$ be a transitive digraph on $n$ vertices with vertex set $V$ and directed edge set $E$. If $w(D) > (3/4)n^2$, then there exists a transitive digraph $D'$ on $n$ vertices such that $w(D)=w(D')$ and $D'$ has at least one mother vertex.
\end{thm}

Theorem \ref{thm:rearrange} allows us to compute $W(n)$ recursively (see Corollary \ref{cor:firstrec}). While this is an interesting result, the real power of Theorem \ref{thm:rearrange} becomes apparent when studying the function $b(n)$. We will prove that $b(n)$ can be determined exactly---and independently of $W(n)$---via its own recursive formula.

\begin{thm}\label{thm:Bigl(z)}
Define $\ell(z):= b(z) - z + 3$.  Let $z \ge  1$ and let $n$ be such that $\ell(z) \le n < \ell(z+1)$. If $n \ne 8$, then $b(n) = n^2 - zn + b(z)$.  
\end{thm}

While the recursive formula is quite effective in practice, for large values of $n$, it requires knowledge of the values of $b(m)$ for $m < n$.  It would be beneficial to have a good estimate for $b(n)$ based only on $n$. Evidently from Theorem \ref{thm:Bigl(z)}, this requires an accurate estimate for the integer $z$ such that $\ell(z) \le n < \ell(z + 1)$.  We are able to provide such an approximation for $z$ (see Definitions \ref{def:zeta}, \ref{def:r} and Theorem \ref{thm:restimate}), which in turn allows us to provide a very good estimate for $b(n)$. 

\begin{thm}\label{thm:bestimate}
Define $N := \lfloor \log_2 \log_5 n \rfloor + 1,$ and define
\begin{equation*}
g(n) :=  n^2 - \left(\sum_{j=1}^N \frac{n^{1 + \frac{1}{2^j}}}{2^{j-1}} \right) + \left(2 - \frac{1}{2^{N-1}}\right)n.
\end{equation*}
For all $n \ge 3$, $|b(n) - g(n)| < 2n$.
\end{thm}

Finally, we are able to use the theory we have built up to obtain an estimate for $|W(n)|$ with an error that is bounded by a constant times $n$, which answers a question of Rao \cite{Rao}.  Note that $|W(n)|$ tells us the number of different integers $k$ for which there exists a transitive digraph on $n$ vertices with exactly $k$ directed edges.

\begin{thm}\label{thm:Westimate}
For each $n \ge 3$, let $N := \lfloor \log_2 \log_5 n \rfloor + 1$, and define
 \[ \oom(n) := n^2 - \sum_{k = 1}^N \left(\frac{2^k}{\prod_{i=1}^k (2^i + 1)}\right) n^{1 + \frac{1}{2^k}}.\] 
For all $n \ge 3$, $||W(n)| - \oom(n)| < 30n$.
\end{thm} 

We remark that $\oom(n)$ is often far closer to $|W(n)|$ than $30n$, and we did not attempt to optimize the constant that is multiplied by $n$ in this inequality.  For example, our methods allow us to determine computationally that $|W(5000)| = 24746694$, whereas 
$\oom(5000) \approx 24752227$, a difference on the order of $5000$.

This paper is organized as follows. In Section \ref{sect:niceform}, we establish some terminology and basic results and prove Theorem \ref{thm:rearrange}.  Section \ref{sect:recursiveb} is devoted to the proof of Theorem \ref{thm:Bigl(z)}. Sections \ref{sect:estimates} and \ref{sect:Westimates} focus on the constructions of the functions $g(n)$ and $\oom(n)$, respectively, and the proofs of Theorems \ref{thm:bestimate} and \ref{thm:Westimate}.  Finally, included in Appendix \ref{app:A} at the end of the paper are data and Mathematica \cite{Mathematica} code used for calculations that occur in certain proofs.   

We point out that the theorems of this paper have applications in other areas.  Recall that a \textit{preorder} or \textit{quasi-order} on a set $S$ is a relation on $S$ that is reflexive and transitive. When $D$ has vertex set $[1,n]$, the set of reachable pairs of $D$ constitutes a preorder on $[1, n]$. Conversely, given a preorder $\mcU$ on $[1,n]$, the directed graph $D$ with vertex set $[1,n]$ and directed edge set $\{(i,j) \in \mcU : i \ne j\}$ is transitive. Thus, the set of reachable pairs of $D$ equals $\mcU$, and finding $W(n)$ is equivalent to determining the possible sizes of a preorder on an $n$-element set.



Preorders also correspond to topologies on finite sets. We sketch this relationship here; details can be found in most references on finite topologies such as \cite{Alex, Erne, Parchmann, RagTen, Sharp, Stanley}. Given a preorder $\mcU$ on $[1, n]$, for each $i \in [1, n]$ let $X_i = \{j \in [1,n] : (i,j) \in \mcU\}$. One may then construct the topology $\mcT$ on $[1,n]$ that has $\{X_i\}_{1\le i \le n}$ as its open basis. Conversely, beginning with a topology $\mcT$ on $[1, n]$, we can recover the corresponding preorder $\mcU$. To do this, for each $i \in [1,n]$ we let $U_i$ be the minimal open set of $\mcT$ containing $i$. Then, $\mcU$ is defined by the rule $(i,j) \in \mcU$ if and only if $U_i \supseteq U_j$. In the topological formulation, the weight corresponds to $\sum_{i=1}^n |X_i| = \sum_{i=1}^n |U_i|$ (or, equivalently, the number of containments $U_i \supseteq U_j$), which is an invariant of $\mcT$. In this way, knowledge of weights and $W(n)$ provides information on these topological spaces and their corresponding preorders.

\section{Elements of the weight set}\label{sect:niceform}

We begin this section with a discussion of basic terminology and ideas that will prove useful.  Let $D$ be a transitive digraph with vertex set $V$ and directed edge set $E$.  A \textit{clique} or \textit{complete directed subgraph} of $D$ is a subset $A \subseteq V$ such that, for all $u \neq v \in A$, there is a directed edge from $u$ to $v$ in $D$. A \textit{mother vertex} is a vertex $u$ such that there is a path of directed edges from $u$ to any other vertex in the digraph.  In a transitive digraph, a mother vertex is a vertex $u$ such that there is a directed edge from $u$ to every other vertex. Given a subset $A \subseteq V$, the \textit{induced subgraph} $D[A]$ has vertex set $A$ and directed edge set $E[A]:=\{(u_1, u_2) \in E : u_1, u_2 \in A \}$.  It follows immediately from transitivity that, if $D$ is a transitive digraph with vertex set $V$, and $A \cup B$ is a partition of $V$, then: both $D[A]$ and $D[B]$ are transitive digraphs; the digraph with vertex set $V$ and directed edge set $E[A] \cup E[B]$ is transitive; and, if $E' := \{(u, v) : u \in A, v \in B\}$, the digraph with vertex set $V$ and directed edge set $E[A] \cup E[B] \cup E'$ is transitive.

We first demonstrate a simple recursive method to produce subsets of $W(n)$.

\begin{lem}\label{Subsets of W(n)}
For each $n \ge 1$, define the following sets $\mathcal{B}_n$ and $\mathcal{C}_n$:
\begin{align*}
\mathcal{B}_n &:= \{n(n-k) + d: d \in W(k), 1 \le k \le n-1\}\\
\mathcal{C}_n &:= \{c+d : c \in W(n-k), d \in W(k), 1 \le k \le n-1\}.
\end{align*}
Then, $W(n) \supseteq \mathcal{B}_n \cup \mathcal{C}_n \cup \{n^2\}$.
\end{lem}
\begin{proof}
Certainly, $n^2 \in W(n)$. The weights in $\mathcal{C}_n$ correspond to a partition $V = A \cup B$ such that $|A|=n-k$; $|B|=k$; and the edge set $E$ of $D$ is the disjoint union of $E[A]$ and $E[B]$ (that is, there are no edges between $A$ and $B$ in either direction). Then, $w(D[A]) \in W(n-k)$, $w(D[B]) \in W(k)$, and $w(D) = w(D[A]) + w(D[B])$. The weights in $\mathcal{B}_n$ correspond to a partition $V = A \cup B$ such that $|A|=n-k$; $|B|=k$; each vertex in $A$ is a mother vertex; no vertex in $B$ is a mother vertex; and there are no edges from a vertex in $B$ to a vertex in $A$. In this case, each vertex in $A$ has weight $n$, $w(D[B]) \in W(k)$, and $w(D)= n(n - k) +w(D[B])$.
\end{proof}

We now provide a coarse lower bound on $b(n)$, where (recalling Definition \ref{def:b}) $b(n)$ is the largest integer such that $[n,b(n)] \subseteq W(n)$.  While Rao's result \cite[Thm.\ 6]{Rao} is better asymptotically, Proposition \ref{prop:3/4b} provides a more convenient lower bound.

\begin{prop}\label{prop:3/4b} 
Let $n \ge 1$ such that $n \ne 7$. Then, $b(n) \ge (3/4)n^2$, and $b(7) = 35$.
\end{prop}
\begin{proof} 
When $n \neq 7$, this follows by inspection for $n \le 18$ (see Appendix \ref{app:2.1}, in which the data is calculated using Lemma \ref{Subsets of W(n)}) and Rao's lower bound $b(n) \ge (n-\lfloor n^{0.57}\rfloor)(n-1) + n$ \cite[Thm.\ 6]{Rao} when $n \ge 19$. For the case $n=7$, we get $b(7) \ge 35$ by inspection (see again Appendix \ref{app:2.1}), and $36 \notin W(7)$ by \cite[p.\ 1597]{Rao}, so in fact $b(7)=35$. 
\end{proof}

Proposition \ref{prop:3/4b} shows that for each $n \ne 7$ and each $k$ between $n$ and $\lceil (3/4)n^2 \rceil$, there always exists a digraph on $n$ vertices with exactly $k$ reachable pairs. However, in general there will be digraphs on $n$ vertices with weight strictly between $\lceil (3/4)n^2 \rceil$ and $n^2$. In Theorem \ref{thm:rearrange}, we will prove that weights in $W(n)$ larger than $(3/4)n^2$ can be realized by digraphs that contain at least one mother vertex, which reduces the problem to considering the induced transitive digraph obtained after removing the mother vertices. Consequently, the weight sets can be determined recursively.

We will now discuss the structure of graphs containing mother vertices.  When a transitive digraph $D$ on a vertex set $V$ has at least one mother vertex, then $V$ admits a partition $V = A \cup B$, where $A$ is the set of all mother vertices.  Assuming that $B$ is nonempty, i.e., assuming that $D$ is not a complete digraph, then there are no directed edges from any vertex in $B$ to any vertex in $A$ (if there were an edge from a vertex $v$ in $B$ to any vertex in $A$, then, since $D$ is a transitive digraph, $v$ would itself be a mother vertex, a contradiction to the definition of $B$).

Of course, when each vertex in $A$ is a mother vertex, $A$ comprises a clique. So, a starting point for studying transitive digraphs with mother vertices is to examine how the vertex set can be partitioned into cliques.

\begin{lem}
 \label{lem:partition}
Let $D$ be a transitive digraph with directed edge set $E$, and let $A \cup B$ be a partition of the vertex set. If both $A$ and $B$ are cliques, then either there is a directed edge from each vertex in $A$ to each vertex in $B$; or, there are no directed edges from any vertex in $A$ to any vertex in $B$.
\end{lem}

\begin{proof}
 Let $u, u^\prime \in A$ and $v, v^\prime \in B$.  If there is a directed edge from $u$ to $v$, then, by transitivity, there is a directed edge from $u^\prime$ to $v^\prime$.  The result follows.
\end{proof}


The following theorem, which is a consequence of \cite[Proposition 2.3.1]{BangJensenGutin}, shows that there is a nice partition of the vertices into cliques. 

\begin{thm}
\label{thm:uppertriangular}
Let $D$ be a transitive digraph.  Then there exists a partition $\{V_i : 1 \le i \le t\}$ of the vertex set such that each $V_i$ is a clique, and the following hold.
\begin{enumerate}[(i)]
\item If $i < j$, then there is no directed edge from any vertex in $V_j$ to any vertex in $V_i$.
\item If $i < j$, then either there is a directed edge from each vertex in $V_i$ to each vertex in $V_j$; or, there are no directed edges from any vertex in $V_i$ to any vertex in $V_j$.
\end{enumerate}
\end{thm}

\begin{proof}
We proceed by induction on the number of vertices of the digraph $D = (V,E)$.  The result is obvious if there are only one or two vertices, so we assume that $n \ge 3$ and that the result is true if there are fewer than $n$ vertices. Let $D$ be a transitive digraph with vertex set $V$ and directed edge set $E$ on $n$ vertices.    Choose a vertex $u$ such that the number of directed edges in $D$ starting at $u$ is maximum.  Define
\begin{equation*} 
V_1 := \{ v \in V : v = u \text{ or } (v, u) \in E\}.
\end{equation*}
First, $V_1$ is nonempty, since $u \in V_1$.  Moreover, if $u^\prime \in V_1$, then by transitivity every out-neighbor of $u$ is an out-neighbor of $u^\prime$.  Since the number of directed edges in $D$ starting at $u$ is a maximum, there must also be a directed edge from $u$ to $u^\prime$, and hence $V_1$ is a clique.  Moreover, if $x \not\in V_1$, then by construction there are no directed edges from $x$ to any vertex in $V_1$.  Thus, $V_1$ is a clique of the desired type, and we may remove the vertices of $V_1$ from $D$ and partition the remaining vertices in the desired fashion by inductive hypothesis.  Using also Lemma \ref{lem:partition}, the result follows.   
\end{proof}

Using the partition guaranteed by Theorem \ref{thm:uppertriangular}, we can bound the weight of $D$ in terms of the size of the largest clique.

\begin{prop}
\label{prop:3/4}
Let $D$ be a transitive digraph on $n$ vertices, and let $\omega$ be the size of the largest clique in $D$. Then, $w(D) \le n(n + \omega)/2$.  In particular, if $\omega \le n/2$, then $w(D) \le (3/4)n^2$.
\end{prop}

\begin{proof}
 Let $D$ be a transitive digraph with partition $\{V_i : 1 \le i \le t\}$ as given by Theorem \ref{thm:uppertriangular} and let $n_i := |V_i|$; note that $\omega = \max\{n_i : 1 \le i \le t \}$.  For a fixed $k$ and all $j < k$, there are no directed edges from any vertex in $V_k$ to any vertex in $V_j$, and, for all $j > k$, there are no directed edges from any vertex in $V_j$ to any vertex in $V_k$.  This means there at least
 \[ n_k\left(\sum_{j \neq k} n_j \right) = n_k(n - n_k)\]
 pairs that are not reachable that include vertices in the clique $V_k$.  If we were to sum over all $k$, we will have counted each such pair twice, so the total number of pairs that are \emph{not} reachable is at least
\begin{equation*}
\frac{1}{2} \sum_{k = 1}^t n_k (n - n_k) = \frac{1}{2} n \sum_{k = 1}^t n_k - \frac{1}{2} \sum_{k = 1}^t n_k^2.
\end{equation*}
Now, $\sum_{k=1}^t n_k = n$ and $n_k \le \omega$ for all $k$, so
\begin{equation*}
\frac{1}{2} n \sum_{k = 1}^t n_k - \frac{1}{2} \sum_{k = 1}^t n_k^2 \ge \frac{1}{2}n^2 - \frac{1}{2} \sum_{k = 1}^t \omega n_k = \frac{1}{2}n^2 - \frac{1}{2}n\omega.
\end{equation*}
Hence, an upper bound on the number of pairs that \emph{are} reachable is $n^2- (\frac{1}{2}n^2 - \frac{1}{2}n\omega) = n(n+\omega)/2$, as desired.
\end{proof}

The converse of Proposition \ref{prop:3/4} tells us that when $w(D) > (3/4)n^2$, there must be a clique of size greater than $n/2$.

We are now able to prove Theorem \ref{thm:rearrange}, which says that any sufficiently large weight in $W(n)$ can be realized by a digraph with at least one mother vertex, and hence the larger values in $W(n)$ can be found by examining $W(k)$ for $k < n$.

\begin{proof}[Proof of Theorem \ref{thm:rearrange}]
The result is obvious if $w(D)=n^2$, so assume that $(3/4)n^2 < w(D) < n^2$ and that the vertices of $D$ are partitioned in the form given by Theorem \ref{thm:uppertriangular}, that is, assume that $\{V_i : 1 \le i \le t\}$ is a partition of the vertex set such that each $V_i$ is a clique and, if $i < j$, there are no directed edges from any vertex in $V_j$ to any vertex in $V_i$ and there is either a directed edge from each vertex in $V_i$ to each vertex in $V_j$ or there are no directed edges from any vertex in $V_i$ to any vertex in $V_j$. Let $V_k$ be the largest clique in $D$, let $\omega = |V_k|$, and let $d=n-\omega$. Since $w(D) > (3/4)n^2$, we have $d < n/2$ by Proposition \ref{prop:3/4}. 

Let $D_1 = D[V \setminus V_k]$; then, $D_1$ is a transitive digraph on $d$ vertices. Consider a vertex $u$ in $V \setminus V_k$. We say that $u$ is \textit{weakly adjacent} to a vertex $v$ if either $(u, v)$ or $(v, u)$ is a directed edge.  If $u$ is weakly adjacent to a vertex in $V_k$, then since $V_k$ is a clique, $u$ is weakly adjacent to every vertex in $V_k$. If $u$ is the initial vertex of a directed edge to some vertex in $V_k$ and is the terminal vertex of a directed edge from some vertex in $V_k$, then $u$ is part of the clique. As this is not the case, we conclude that if $u$ is weakly adjacent to a vertex in $V_k$, then there are exactly $\omega = |V_k|$ directed edges between $u$ and $V_k$. 

Let $c$ be the number of vertices in $V \setminus V_k$ that are weakly adjacent to a vertex in $V_k$. Then, there are exactly
\begin{equation*}
\omega^2 + \omega c = (n-d)^2 + (n-d)c
\end{equation*}
reachable pairs of $D$ that contain a vertex from $V_k$. Thus,
\begin{equation*}
w(D) = (n-d)^2 + (n-d)c + w(D_1).
\end{equation*}

Next, since $c \le d < n/2$, we have $d-c < n-d$. Let $B$ be any subset of $V$ obtained by deleting $(n - d) - (d - c)$ vertices from $V_k$, and let $D_2=D[B]$. Then, $D_2$ is transitive and has $n - (n - 2d + c) = 2d - c$  vertices, and 
\begin{equation*}
w(D_2) = (d-c)^2 + (d-c)c + w(D_1).
\end{equation*}
Some basic manipulation then shows that
\begin{align*}
w(D) &= (n-d)^2 + (n-d)c + w(D_1)\\
&= n(n -2d + c) + (d-c)^2 + (d-c)c + w(D_1)\\
&= n(n - (2d - c)) + w(D_2).
\end{align*}
Thus, we may form the transitive digraph $D'$ on $n$ vertices by starting with the transitive digraph $D_2$ on $2d - c$ vertices and adding $n - (2d -c)$ mother vertices. The vertex set of $D'$ is $A \cup B$, where $A$ contains the newly added mother vertices and $B$ is the vertex set of $D_2$. Then, $w(D') = w(D)$, and $D'$ has at least one mother vertex.
\end{proof}

When $D$ has at least one mother vertex, the weight of $D$ depends entirely on the weight of the subgraph induced by the non-mother vertices. Thus, we now have a method to determine the sets $W(n)$ recursively.  

\begin{cor}\label{cor:firstrec}
Let $n \ge 1$ and let $\mathcal{B}_n := \{n(n-k) + d : d \in W(k), 1 \le k \le n-1\}$. If $n = 7$, then
\begin{equation*}
W(7) = [7, 35] \cup \mathcal{B}_7 \cup \{49\},
\end{equation*}
and if $n \ne 7$ then
\begin{equation*}
W(n) = [n, \lceil (3/4)n^2 \rceil] \cup \mathcal{B}_n \cup \{n^2\}.
\end{equation*}
\end{cor}

\begin{proof}
We know that $\mathcal{B}_n \cup \{n^2\} \subseteq W(n)$ for all $n$ from Lemma \ref{Subsets of W(n)}. By Proposition \ref{prop:3/4b}, we have $[n, \lceil (3/4)n^2 \rceil ] \subseteq W(n)$ for $n \ne 7$ and $b(7) = 35$. So,  
\begin{equation*}
W(7) \supseteq [7, 35] \cup \mathcal{B}_7 \cup \{49\},
\end{equation*}
and when $n \ne 7$,
\begin{equation*}
W(n) \supseteq [n, \lceil (3/4)n^2 \rceil ] \cup \mathcal{B}_n \cup \{n^2\}.
\end{equation*}
 
For the reverse containments, first assume that $n \neq 7$ and that $D$ is a transitive digraph on $n$ vertices that is not a complete graph.  If $w(D) \le (3/4)n^2$, then $w(D) \le b(n)$, and we are done.  Otherwise, $w(D) > (3/4)n^2$, and by Theorem \ref{thm:rearrange}, $w(D) = w(D^\ast)$ for some digraph $D^\ast$ whose number of reachable pairs is an element of $\mathcal{B}_n$, and we are done.

When $n = 7$, we have $(3/4) \cdot 7^2 = 36.75$, so the arguments of the previous paragraph apply to transitive digraphs with at least $37$ reachable pairs.  We also know that $[7, 35] \subseteq W(7)$, so the only remaining question is whether there is a transitive digraph on $7$ vertices with exactly $36$ reachable pairs.  However, this is ruled out by computations performed in \cite[p.\ 1597]{Rao}. Hence, $W(7) = [7, 35] \cup \mathcal{B}_7 \cup \{49\}$, and the proof is complete.
\end{proof}

\section{A recursive formula for $b(n)$}\label{sect:recursiveb}

In this section, we demonstrate the power of Theorem \ref{thm:rearrange} and more closely examine the function $b(n)$, which is equal to the largest positive integer such that $[n, b(n)] \subseteq W(n)$, and thus bounds the beginning values in $W(n)$. By Proposition \ref{prop:3/4b} and Theorem \ref{thm:rearrange}, we know that $b(n) \ge (3/4)n^2$ for all $n \ne 7$, and that any number of reachable pairs greater than $(3/4)n^2$ can be realized via a digraph $D$ with vertex set $V = A \cup B$, where $A$ is the nonempty set of mother vertices of $V$.  We note also that the induced subgraph $D[B]$ is a transitive digraph on fewer than $n$ vertices.  By examining transitive digraphs with mother vertices, we can find integers that are in $[n, n^2] \setminus W(n)$, and since $b(n)+1$ is the smallest such integer, this gives us useful information about $b(n)$. 

Let us consider one way in which a gap could appear in the weight set $W(n)$. If $|B|=z$, then it is possible that $w(D) = n^2-zn+b(z)$, but $w(D) \ne n^2-zn+b(z) +1$, because $w(D[B]) \ne b(z)+1$. If, in addition, $n^2-zn+b(z)+1$ is not the weight of any digraph with $z-1$ non-mother vertices, then it is plausible that $n^2-zn+b(z)+1 \notin W(n)$. This will occur if 
\begin{equation*}
n^2-zn+b(z)+2 \le n^2-(z-1)n+z-1.
\end{equation*}
Solving this inequality for $n$ yields $n \ge b(z)-z+3$. This inspires the next definition.

\begin{Def}\label{def:weightloss}
For all $z \in \N$, we define $\ell(z) := b(z) - z + 3$. 
\end{Def}
 
Clearly, we must know $b(z)$ to be able to calculate $\ell(z)$. However, it turns out that knowing $\ell(z)$ allows us to compute $b(n)$ for some values of $n \ge \ell(z)$. The relationship (barring some small exceptions) between $b(n)$ and $\ell(z)$ is that if $\ell(z) \le n < \ell(z+1)$, then $b(n) = n^2 - zn + b(z)$. Thus, if we know $b(k)$ for $k \in [1, n-1]$, then we can calculate each $\ell(k)$, find the appropriate $z$, and then calculate $b(n)$. See Appendix \ref{app:1.4}, which provides pseudocode and Mathematica code for computing $b$, $\ell$, and $z$, and lists the values of these functions for $1 \le n \le 24$.

Most of this section is dedicated to proving the statements of the previous paragraph. The main theorem is Theorem \ref{thm:Bigl(z)}, and the majority of the work is done in Propositions \ref{prop:b(n)lowerbound} and \ref{prop:Usefulprop}. Computational lemmas are introduced as they are needed to prove the propositions.

\begin{lem}\label{lem:FirstLlemma}
For all $n \ge 1$ and all $m \ge 1$, $b(n) + m \le b(n+m)$.
\end{lem}
\begin{proof}
For any $n \ge 1$, we can form a transitive digraph $D$ on $n+m$ vertices by taking a transitive digraph $D^\prime$ on $n$ vertices with between $n$ and $b(n)$ reachable pairs and adding $m$ isolated vertices, yielding transitive digraphs on $n+m$ vertices with between $n+m$ and $b(n) + m$ reachable pairs.   The result follows.
\end{proof}

\begin{lem}\label{lem:FourthLlemma}
Let $z \ge 6$. Then, 
\begin{enumerate}[(1)]
\item $\ell(z) > 4z$,
\item if $n \ge \ell(z)$, then $n^2 - zn + b(z) \ge (3/4)n^2$.
\end{enumerate}
\end{lem}
\begin{proof}
(1) We have $\ell(6) = 25$ and $\ell(7) = 31$, so assume that $z \ge 8$. By Proposition \ref{prop:3/4b}, $b(z) \ge (3/4)z^2$, so
\begin{equation*}
\ell(z) = b(z) - z + 3 \ge (3/4)z^2 - z + 3
\end{equation*}
and it is routine to verify that this is greater than $4z$.

For (2), assume $n \ge \ell(z)$ and let $h(x) = (1/4)x^2 - zx + b(z)$. Solving $h(x) = 0$ for $x$ in terms of $z$ yields $x = 2(z \pm \sqrt{z^2 - b(z)})$. We have
\begin{equation*}
2(z + \sqrt{z^2 - b(z)}) < 4z < \ell(z) \le n
\end{equation*}
so the desired inequality holds.
\end{proof}

\begin{prop}\label{prop:b(n)lowerbound}
Let $z \ge  4$ and let $n$ be such that $\ell(z) \le n < \ell(z+1)$. Then, $b(n) \ge n^2 - zn + b(z)$.
\end{prop}
\begin{proof}
Since $z \ge 4$, $n \ge \ell(4) = 12$, so by Proposition \ref{prop:3/4b}, $b(n) \ge (3/4)n^2$. Moreover, by Lemma \ref{lem:FourthLlemma}(2), $n^2 - zn + b(z) \ge (3/4)n^2$.  To get the stated result, we need to show that there exists a digraph on $n$ vertices with exactly $m$ reachable pairs for every integer $m$ satisfying $(3/4)n^2 \le m \le n^2 - zn + b(z)$. 

Consider transitive digraphs on $n$ vertices that have exactly $n-k$ mother vertices, and let $D_k$ be the subgraph on $k$ vertices induced by the set of non-mother vertices. We can vary the choice of $D_k$ so that the number of reachable pairs in $D_k$ is any integer between $k$ and $b(k)$ (inclusive). This allows us to produce transitive digraphs on $n$ vertices with numbers of reachable pairs in the interval
\begin{equation*}
I_k := [n(n-k) + k, \quad n(n-k) +b(k)],
\end{equation*}
and every weight in $I_k$ is achievable. We will prove that the union $\bigcup_{k=z}^{n-1} I_k$ covers the entire interval $[\lceil (3/4)n^2 \rceil, n^2 - zn + b(z)]$. (Note that, if we were to order the intervals by their left endpoint, then, by our choice of notation, $I_{n-1}$ is the leftmost interval and $I_z$ is the rightmost interval.)

We claim that when $z+1 \le k$, the lower endpoint of $I_{k-1}$ is at most one more than the upper endpoint of $I_k$. That is, we seek to show that
\begin{equation}\label{endpoint ineq}
n(n-(k-1))+(k-1) \le 1 + n(n-k) + b(k)
\end{equation}
which is equivalent to showing 
\begin{equation*}
n+k-1 \le b(k) + 1.
\end{equation*}
Now, by assumption, $n \le \ell(z+1) - 1$, and by definition $\ell(z+1) = b(z+1) - (z+1) + 3$. From this, we get that $n+z \le b(z+1)+1$. Using Lemma \ref{lem:FirstLlemma}, we have 
\begin{align*}
n+k-1 &= n + z + (k - z - 1)\\
&\le b(z+1)+1+ \left(k - (z + 1)\right)\\
&\le b(z+1+ k - (z + 1))+1\\
&= b(k) + 1.
\end{align*}
Thus, \eqref{endpoint ineq} holds. This means that the union of the intervals $I_k$ comprises a single interval, and, in particular, contains the interval $[2n-1, n^2-zn+b(z)]$, which goes from the lower endpoint of $I_{n-1}$ to the upper endpoint of $I_z$. Clearly, $2n-1 \le (3/4)n^2$, so we conclude that
\begin{equation*}
[n, \; \lceil (3/4)n^2 \rceil] \cup [2n-1, \;\; n(n-z)+b(z)] = [n, \;\; n^2-zn+b(z)]
\end{equation*}
and the stated result follows.
\end{proof}

\begin{lem}\label{lem:ThirdLlemma}
Assume that $z \ge 4$, $n \ge \ell(z)$, and $z+1 \le d < n/2$. Then, 
\begin{equation*}
n^2-dn+d^2 \le n^2-zn+b(z).
\end{equation*}
\end{lem}

Before we prove the result, we note that we need $z$ to be at least $4$ in order to have $\ell(z)/2 > z+1$. If $z = 3$, then $\ell(z) = 7$, and $z+1 = 4 > 7/2$. However, once $z \ge 4$, one can use Proposition \ref{prop:3/4b} to easily show that $\ell(z) > 2z+2$. Hence, our assumption on $z$, given that $n \ge \ell(z)$ and $z+1 \le d < n/2$, is necessary and can be satisfied.

\begin{proof}
First, let $h(d) = n^2 - dn + d^2$. Then, $h$ is decreasing for $d < n/2$, so
\begin{equation}\label{f ineq}
n^2 - dn + d^2 \le n^2 - (z+1)n + (z+1)^2.
\end{equation}
Next, we claim that $2b(z) \ge z^2 + 3z - 2$. This is clear when $4 \le z \le 7$, and for $z \ge 8$ one may apply Proposition \ref{prop:3/4b} and verify that
\begin{equation*}
2b(z) \ge (3/2)z^2 \ge z^2 + 3z - 2 = (z+1)^2 + z-3,
\end{equation*}
which implies that
\begin{equation}\label{n ineq}
n  + b(z) \ge \ell(z) + b(z) = 2b(z) - z + 3 \ge (z+1)^2.
\end{equation} 
Combining \eqref{n ineq} and \eqref{f ineq} yields
\begin{align*}
n^2 - dn + d^2 \le n^2 - (z+1)n + (z+1)^2 \le n^2 -zn + b(z),
\end{align*}
as desired.
\end{proof}

The next proposition strengthens Theorem \ref{thm:rearrange}, and is the key to the proof of Theorem \ref{thm:Bigl(z)}.

\begin{prop}\label{prop:Usefulprop}
Let $z \ge 6$ and let $n$ be such that $n \ge \ell(z)$. Let $D$ be a transitive digraph on $n$ vertices with at least $n^2-zn + b(z) + 1$ reachable pairs. Then, there exists a transitive digraph $D^\prime$ such that $w(D^\prime) = w(D)$ and $D^\prime$ has vertex set $A \cup B$, where $A$ is the set of mother vertices in $D^\prime$, $B$ is the set of non-mother vertices, and $|B|=s$ for some $0 \le s \le z$.
\end{prop}

\begin{proof}
The proposition is obvious if $D$ is complete, so assume that $w(D) < n^2$. Since $z \ge 6$, we have $w(D) > n^2 -zn + b(z) \ge (3/4)n^2$ by Lemma \ref{lem:FourthLlemma}(2). Hence, Theorem \ref{thm:rearrange} can be applied.  Let $V_k$ be the largest clique of $D$ with $n_k := |V_k|$, $d := n - n_k$, and, given $u \in V_k$, define 
\begin{equation*}
c:= |\{x \in V : x \not\in V_k, (u, x) \text{ or } (x, u) \in E \}|,
\end{equation*}
that is, $c$ is the number of vertices in $V \setminus V_k$ that are weakly adjacent to a vertex in $V_k$. From the proof of Theorem \ref{thm:rearrange}, we know that there is a transitive digraph $D'$ such that $w(D')=w(D)$ and the vertex set of $D'$ can be partitioned into $A \cup B$, where $A$ is the set of all mother vertices in $D'$, $B$ is the set of non-mother vertices, and $B$ consists of $s:=2d-c$ vertices. It remains to show that $s \le z$.

As noted above, $w(D) > (3/4)n^2$, so $n_k  > n/2$ by Proposition \ref{prop:3/4} and hence $d = n - n_k < n/2$. Suppose that $z+1 \le d < n/2$. Then, the maximum number of reachable pairs in $D$ is
\begin{equation*}
w(D) \le n(n-d) + d^2,
\end{equation*}
which by Lemma \ref{lem:ThirdLlemma} is strictly less than $n^2 - zn + b(z) + 1$. So, we must have $d \le z$.

By Lemma \ref{lem:FourthLlemma}(1), $n > 4z$ and we are assuming that $d \le z$, so we get $d < n/4$. This gives $n-2d+c > n/2$ and $s < n/2$. If $z+1 \le s < n/2$, then using Lemma \ref{lem:ThirdLlemma} shows that 
\begin{align*}
w(D^\prime) &= n(n-2d+c) + w(D'[B])\\
&\le n(n-s) + s^2\\
&< n^2 - zn + b(z) + 1.
\end{align*}
This contradicts the fact that $w(D^\prime) = w(D)$. Thus, $s \le z$, as required.
\end{proof}

We are now ready to prove Theorem \ref{thm:Bigl(z)}.

\begin{proof}[Proof of Theorem \ref{thm:Bigl(z)}] 
For $z \in [1, 5]$, the theorem can be proved by inspection using Corollary \ref{cor:firstrec} and Table \ref{W table}; see Appendix \ref{app:1.4}. Note that the case $n=8$ occurs when $z=3$; see the remark following this proof for an explanation of why the result does not hold in this instance.

Assume that $z \ge 6$. We know that $b(n) \ge n^2 - zn + b(z)$ by Proposition \ref{prop:b(n)lowerbound}. Let $D$ be a transitive digraph on $n$ vertices with $w(D) > n^2 - zn + b(z)$. By Proposition \ref{prop:Usefulprop}, we may assume that the vertex set of $D$ is $A \cup B$, where $A$ is the set of all mother vertices in $D$, $B$ is the set of non-mother vertices, and $|B|=s$ for some $0 \le s \le z$. We will argue that $w(D) \ge n^2 - zn + b(z)+2$.

We have $w(D) = n(n-s)+w(D[B])$. If $s=z$, then $w(D[B]) > b(z)$. As $b(z)+1 \notin W(z)$, we must have $w(D[B]) \ge b(z)+2$ and we are done. So, assume that $s \le z-1$. Then, a lower bound for $w(D)$ is
\begin{equation}\label{morty}
w(D) = n(n-s) + w(D[B]) \ge n(n-s) + s.
\end{equation}
Now, $n(n-s) + s$ is decreasing as a function of $s$, so \eqref{morty} and the fact that $n \ge \ell(z)$ yields
\begin{align*}
w(D) &\ge n(n-(z-1)) + z-1 \\
&= n^2 - zn + z -1 + n\\
&\ge n^2 - zn + z - 1 + (b(z) - z + 3)\\
&= n^2 - zn + b(z) + 2.
\end{align*}
Hence, $n^2 - zn + b(z) +1 \notin W(n)$ and therefore $b(n) = n^2 - zn + b(z)$.
\end{proof}

\begin{Rem}\label{rem:n=8remark}
When $n=8$, the corresponding $z$ is $z=3$, for which $\ell(z) = 7$. Computing $n^2-zn + b(z)$ in this case yields $47$, but $b(8) = 52$ by Table \ref{W table}.  What differs in this situation is that, when $n = 8$, \[(3/4)n^2 = n^2-zn + b(z) + 1,\] and there exists a digraph on $8$ vertices of weight 48: namely, start with a complete digraph on $4$ vertices and add to it exactly $4$ mother vertices.  Lemma \ref{lem:FourthLlemma}(2) shows that such coincidences cannot happen when $z$ (and $n$) are sufficiently large. 
\end{Rem}

Our first application of Theorem \ref{thm:Bigl(z)} is to use it to improve Corollary \ref{cor:firstrec} for $n \ge 8$, since we now need only check the values of $k$ through $z$ (as opposed to $n-1$).  

\begin{cor}\label{cor:secondrec}
Let $n \ge 8$ and let $\mathcal{B}_n := \{n(n-k) + d : d \in W(k), 1 \le k \le z\}$, where $z$ is such that $\ell(z) \le n < \ell(z+1)$. Then, 
\begin{equation*}
W(n) = [n, b(n)] \cup \mathcal{B}_n \cup \{n^2\}.
\end{equation*}
\end{cor}

\begin{proof}
 Certainly, $[n, b(n)] \cup \mathcal{B}_n \cup \{n^2\} \subseteq W(n)$ by Corollary \ref{cor:firstrec}.  For $8 \le n \le 11$, the result follows by direct calculation and comparing the sets listed in Corollaries \ref{cor:firstrec} and \ref{cor:secondrec}.  By Proposition \ref{prop:3/4b} and Theorem \ref{thm:Bigl(z)}, $b(n) = n^2 - zn + b(z) \ge (3/4)n^2$ when $n \ge 12$, and, by Proposition \ref{prop:3/4}, if $w(D) > (3/4)n^2$, this means $\omega > n/2$.  Thus, if $w(D) > (3/4)n^2$, then $w(D) = n(n - k) + d$ for some $k < n/2$ and $d \in W(k)$.  Since $n \ge 12$, we have $z \ge 4$, so, by Lemma \ref{lem:ThirdLlemma}, if $z+1 \le k < n/2$ and $w(D) = n(n - k) + d$ for some $k < n/2$ and $d \in W(k)$, we have
 \[w(D) \le n(n-k) + k^2 \le n^2 - zn + b(z) = b(n),\] and so we only need to check values of $k$ that are at most $z$.  The result follows.  
\end{proof}

\section{Estimating $b(n)$}\label{sect:estimates}


The purpose of this section is to estimate $b(n)$: that is, given $n$, can we get a reasonably accurate estimate of $b(n)$ without any knowledge of $b(m)$ for $m < n$? In Section \ref{sect:Westimates}, we will consider the same question for $|W(n)|$. By Corollary \ref{cor:secondrec}, such an estimate for $b(n)$ would be useful toward finding an estimate for $|W(n)|$, and, in light of Theorem \ref{thm:Bigl(z)}, finding an estimate for $b(n)$ more or less reduces to having an accurate estimate for the integer $z$ such that $\ell(z) \le n < \ell(z+1)$.

\begin{Def}\label{def:zeta}
For each $n \in \N$, $n \ge 3$, we define $\zeta(n)$ to be the unique positive integer such that $\ell(\zeta(n)) \le n < \ell(\zeta(n)+1)$.
\end{Def}

For computation of some values of $\zeta(n)$ for small values of $n$, see Appendix \ref{app:1.4}. 

We will often need to iterate the function $\zeta$. When $n$ is clear from context, we let $z_1 := \zeta(n)$ and $z_k:=\zeta^k(n)=\zeta(z_{k-1})$ for each $k \ge 2$. With this notation, $z_1$ has the same definition as the integer $z$ that appeared throughout Section \ref{sect:recursiveb}.

The majority of this section is devoted to developing an explicit function $r(n)$ to estimate $\zeta(n)$ (Definition \ref{def:r}), and proving that $r$ does in fact accurately approximate $\zeta$ (Theorem \ref{thm:restimate}). Our first lemma lists some simple observations about the functions $b$, $\ell$, and $\zeta$. All of these follow from Theorem \ref{thm:Bigl(z)} and the definitions of the functions, and we shall use them freely in our subsequent work.

\begin{lem}\label{lem:easyobs}
Let $x \in \N$, $x \ge 9$. Then,
\begin{enumerate}[(1)]
\item $b(x) = x^2 - \zeta(x)x + b(\zeta(x))$.
\item $\ell(x) = x^2 - (\zeta(x)+1)x+b(\zeta(x))+3$.
\item If $x < \ell(\zeta(x)+1)-1$, then $\zeta(x+1)=\zeta(x)$. If $x = \ell(\zeta(x)+1)-1$, then $\zeta(x+1)=1+\zeta(x)$.
\end{enumerate}
\end{lem}

In order to approximate $b(n)=n^2-z_1n+b(z_1)$ in terms of $n$, we require bounds on $z_1$ and $b(z_1)$. These are obtained below in Lemmas \ref{lem:betterzbounds} and \ref{lem:bandl}(3), respectively. Usually, small cases must be checked by hand; this can be accomplished by using Theorem \ref{thm:Bigl(z)} to calculate $b$ and $\ell$ recursively.

\begin{lem}\label{lem:betterzbounds}
When $n \ge 12$, $\sqrt{n} \le z_1 < \sqrt{2n}$, and the lower bound is strict for $n \ne 16$.
\end{lem}
\begin{proof}
When $12 \le n \le 46$, the lower bound holds by inspection (see Appendix \ref{app:4.3}), and $\sqrt{n}=\zeta(n)$ only when $n=16$. So, assume $n \ge 47$ (which implies that $z_1 \ge 8$) and let $t := \zeta(z_1 + 1)$; then, $t \ge 3$. One may compute that $\ell(z_1+1)=z_1^2-(t-1)z_1+\ell(t)$. Since $n < \ell(z_1+1)$ and $\ell(t) \le z_1 + 1$, we obtain
\begin{equation*}
n < z_1^2 - (t-1)z_1 + \ell(t) \le z_1^2 - (t-1)z_1 + z_1 + 1 < z_1^2
\end{equation*}
where the last inequality holds because $t \ge 3$.

For the upper bound, we apply induction. The bound holds for $12 \le n \le 99$ by inspection, so assume $n \ge 100$, which means that $z_1 \ge 12$. Using this and the inductive hypothesis, we see that
\begin{equation*}
2(z_2 + 1) < 2(\sqrt{2z_1} +1 ) < z_1.
\end{equation*}
Now, since $n \ge \ell(z_1)$, we have
\begin{equation*}
n \ge \ell(z_1) = z_1^2 - (z_2+1)z_1 + b(z_2) + 3 > z_1^2 - (z_2 + 1)z_1 > z_1^2 - \frac{z_1^2}{2}.
\end{equation*}
The result follows. 
\end{proof}

\begin{lem}\label{lem:bandl} 
Let $n \in \N$.
\begin{enumerate}[(1)]
\item For all $n \ge 8$, $n+1 \le b(n+1) - b(n) \le 2n + 1$.
\item If $n \ge 4$, then $\ell(n) < b(n) < \ell(n+1)$,
\item If $n \ge 192$, then $|n - b(z_1)| \le z_1 - 3$.
\end{enumerate}
\end{lem}

\begin{proof}
For (1), we use induction. When $n \le 46$, the lemma holds by inspection (see Appendix \ref{app:4.4(1)}). If $n \ge 47$ and $n < \ell(z_1+1)-1$, then  $b(n+1)-b(n)=2n+1-z_1$ and the bounds hold. Finally, if $n \ge 47$ and $n = \ell(z_1+1)-1$, then one may compute that $b(n+1)-b(n) = n-z_1 + b(z_1+1)-b(z_1)$ and apply the inductive hypothesis to obtain the desired bounds.

Part (2) is true by inspection for $4 \le n \le 8$ (see Appendix \ref{app:1.4}), and follows from part (1) and the definition of $\ell$ for $n \ge 9$. For (3), the bounds hold by inspection for $192 \le n \le 480$ (see Appendix \ref{app:4.4(3)}). So, assume that $n \ge 481$, which means that $z_1 \ge 25$ and $z_2 \ge 6$. If $n \le b(z_1)$, then the results follows from the definition of $\ell(z_1)$. So, assume that $n > b(z_1)$. We claim that $b(z_1+1)-b(z_1) \le 2z_1+1-z_2$. Indeed, the two expressions are equal when $z_1 < \ell(z_2+1)-1$, and if $z_1 = \ell(z_2+1)-1$, then one may check that
\begin{equation*}
b(z_1+1)-b(z_1) = 2z_1+1-z_2 - (\ell(z_2)-1) < 2z_1+1-z_2.
\end{equation*}
Using this inequality and the fact that $n < \ell(z_1+1) = b(z_1+1)-(z_1+1)+3$, we obtain
\begin{equation*}
n-b(z_1) < b(z_1+1) - b(z_1) -(z_1+1)+3 \le 2z_1+1-z_2 - (z_1+1)+3 \le z_1-3,
\end{equation*}
as desired.
\end{proof}

At this point, we know that (for sufficiently large $n$), $n \approx b(z_1)$ and $z_1 \approx b(z_2)$. So,
\begin{equation*}
n \approx b(z_1) = z_1^2 - z_2z_1 + b(z_2) \approx z_1^2 - (z_2 - 1)z_1.
\end{equation*}
From this, we obtain the estimate $z_1 \approx n^\frac{1}{2} + \frac{1}{2}(z_2-1)$, which suggests that
\begin{equation*}
\zeta(n) \approx n^\frac{1}{2} + \frac{1}{2}\big(\zeta(n^{\frac{1}{2}}) - 1\big).
\end{equation*}
Thus, to approximate $\zeta(n)$, we should construct a function $r(n)$ that satisfies
\begin{equation}\label{eq:rcondition}
r(n) = n^\frac{1}{2} + \frac{1}{2} \big(r(n^\frac{1}{2}) - 1 \big). 
\end{equation}

\begin{Def}\label{def:r}
Given a real number $n \ge 3$, define $N := \lfloor \log_2 \log_5 n \rfloor + 1,$ and define
\begin{equation*}
r(n) := \left(\sum_{j=1}^N \frac{1}{2^{j-1}} n^{\frac{1}{2^j}}\right) - \frac{2^{N-1} - 1}{2^{N-1}}.
\end{equation*}
\end{Def}  

We briefly discuss our motivation for our choice of $N$.  Supposing that $r(n) \approx z_1$, $r(n^\frac{1}{2}) \approx z_2$, etc., the question then becomes how often we would need to iterate until $z_N = 1$.  If $z_N \approx n^\frac{1}{2^N}$, then this implies that $N$ is the least integer such that $n^\frac{1}{2^N} < \ell(2) = 5$, i.e., $N = \lfloor \log_2 \log_5 n \rfloor + 1.$  Moreover, direct calculation shows that $r(n)$ does satisfy \eqref{eq:rcondition}, and we will prove that $r(n)$ is a good approximation for $\zeta(n)$.

\begin{thm} \label{thm:restimate}
For all $n \ge 3$, $|\zeta(n) - r(n)| < 1.985$.
\end{thm}

To prove Theorem \ref{thm:restimate}, we use the triangle inequality and \eqref{eq:rcondition} multiple times to split $|\zeta(n) - r(n)|$ into six summands, each of which can be bounded individually. Letting $t = \zeta(\lceil n^\frac{1}{4} \rceil)$, one may obtain
\begin{align}\label{eq:zrdiff}
\begin{split}
|\zeta(n) - r(n)| = & \left|z_1 - \left( n^\frac{1}{2} + \frac{r(n^\frac{1}{2}) - 1}{2}\right)\right|\\
\le & \left|z_1 - \left(n^\frac{1}{2} + \frac{z_2 - 1}{2} \right) \right| + \frac{1}{2}\left| z_2 - r(n^\frac{1}{2})\right|\\
= & \left|z_1 - \left(n^\frac{1}{2} + \frac{z_2 - 1}{2} \right) \right| + \frac{1}{2}\left| z_2 - \left(n^\frac{1}{4} + \frac{r(n^\frac{1}{4}) - 1}{2} \right)\right|\\
\le & \left|z_1 - \left(n^\frac{1}{2} + \frac{z_2 - 1}{2} \right) \right|
+ \frac{1}{2}\left|z_2 - \left(z_1^\frac{1}{2} 
+ \frac{z_3 - 1}{2} \right) \right| 
+ \frac{1}{2}\left|z_1^\frac{1}{2} - n^\frac{1}{4} \right|\\
& + \frac{1}{4}\left| z_3 - r(n^\frac{1}{4})\right|\\
\le & \left|z_1 - \left(n^\frac{1}{2} + \frac{z_2 - 1}{2} \right) \right|
+ \frac{1}{2}\left|z_2 - \left(z_1^\frac{1}{2} 
+ \frac{z_3 - 1}{2} \right) \right| 
+ \frac{1}{2}\left|z_1^\frac{1}{2} - n^\frac{1}{4} \right|\\
& + \frac{1}{4}\left|z_3 -t\right| 
+ \frac{1}{4}\left|t - r(\lceil n^\frac{1}{4} \rceil)\right| 
+ \frac{1}{4}\left| r(\lceil n^\frac{1}{4} \rceil) - r(n^\frac{1}{4}) \right|. 
\end{split}
\end{align} 
We now present a series of lemmas that allow us to bound each of these terms.

\begin{lem}\label{lem:zrootndiff}
When $n \ge 33808$, 
\begin{equation*}
\left|z_1 - \left(n^{\frac{1}{2}}+ \frac{z_2 - 1}{2}\right)\right| < \frac{\sqrt{2}}{\sqrt{z_1}} + \frac{\frac{5}{4} - \frac{3}{z_1} + 
(\frac{1}{8z_1})^\frac{1}{4}}{2\sqrt{1 - \frac{\sqrt{2}}{\sqrt{z_1}} - \frac{1}{z_1}}}.
\end{equation*}
\end{lem}
 
\begin{proof}
First, the condition on $n$ implies that $z_1 \ge 192$, so by Lemma \ref{lem:bandl}(3) we have $|n - b(z_1)| \le z_1 - 3$ and $|z_1-b(z_2)| \le z_2-3$.  Second, from $b(z_1) = z_1^2 - z_2z_1 + b(z_2)$, we obtain
\begin{equation}\label{eq:beforec}
b(z_1) + \frac{(z_2 - 1)^2}{4} = \left(z_1 - \frac{z_2 - 1}{2} \right)^2 + (b(z_2) - z_1).
\end{equation}
For readability, let $c = b(z_1) + \tfrac{1}{4}(z_2-1)^2$. Then, \eqref{eq:beforec} yields
\begin{equation*}
\left| \left(z_1 - \frac{z_2 - 1}{2} \right)^2 - c\right| \le |z_1 - b(z_2)| \le z_2 - 3.
\end{equation*}
Next, we use the triangle inequality, rationalize the numerators, and use the bounds we have already discussed in this proof:
\begin{align*}
\left|z_1 - \left(n^{\frac{1}{2}}+ \frac{z_2 - 1}{2}\right)\right| 
&\le \left|\left(z_1 - \frac{z_2 - 1}{2}\right) - \sqrt{c}\right| 
+ \left|\sqrt{c} - n^\frac{1}{2}\right|\\
&= \frac{\left|\left(z_1 - \frac{z_2 - 1}{2} \right)^2 - c\right|}{\left(z_1 - \frac{z_2 - 1}{2}\right) + \sqrt{c}} 
+ \frac{\left|c - n \right|}{\sqrt{c} + n^\frac{1}{2}}\\
&\le \frac{z_2 - 3}{\left(z_1 - \frac{z_2 -1}{2}\right) + \sqrt{c}} + \frac{(z_1 - 3) + \frac{(z_2 - 1)^2}{4}}{\sqrt{c} + n^{\frac{1}{2}}}.
\end{align*} 
Using Lemma \ref{lem:betterzbounds}, we have $\sqrt{z_1} < z_2 < \sqrt{2z_1}$, and so we bound the first term by
\begin{equation*}
\frac{z_2 - 3}{\left(z_1 - \frac{z_2 -1}{2}\right) + \sqrt{c}} < \frac{z_2}{z_1} < \frac{\sqrt{2}}{\sqrt{z_1}}.
\end{equation*}
Note that $z_1 \ge \ell(z_2) > z_2^2 - (z_3 + 1)z_2$. Via Lemma \ref{lem:betterzbounds}, this gives $(z_2 - 1)^2 < z_1 + z_2z_3 < z_1 + 2^\frac{5}{4}z^\frac{3}{4}$. Moreover, both $c$ and $n$ are bounded below by $\ell(z_1) > z_1^2-(z_2+1)z_1$. By applying these bounds and canceling factors of $z_1$, we get
\begin{equation*}
\frac{(z_1 - 3) + \frac{(z_2 - 1)^2}{4}}{\sqrt{c} + n^{\frac{1}{2}}} 
< \frac{(z_1 - 3) + \frac{z_1}{4} + \frac{z_2z_3}{4}}{2\sqrt{z_1^2 - (z_2 + 1)z_1}} < \frac{\frac{5}{4} - \frac{3}{z_1} + 
(\frac{1}{8z_1})^\frac{1}{4}}{2\sqrt{1 - \frac{\sqrt{2}}{\sqrt{z_1}} - \frac{1}{z_1}}},
\end{equation*}
which completes the proof.
\end{proof}
 
\begin{lem}\label{lem:rootzrootrootndiff}
When $n \ge 3350194786$, $\frac{1}{2}|z_1^\frac{1}{2} - n^\frac{1}{4} | < 0.13$.
\end{lem}

\begin{proof}
When $n \ge 3350194786$, we have $z_1 \ge 58006$, and so $|z_1 - n^\frac{1}{2}| < \frac{z_2 - 1}{2} + 0.652$ by Lemma \ref{lem:zrootndiff}, which implies that
\begin{align*}
\frac{1}{2}\left|z_1^\frac{1}{2} - n^\frac{1}{4} \right| 
= \frac{|z_1 - n^\frac{1}{2}|}{2|z_1^\frac{1}{2} + n^\frac{1}{4}|} 
< \frac{\frac{z_2 - 1}{2} + 0.652}{2|z_1^\frac{1}{2} + n^\frac{1}{4}|} 
\le \frac{\frac{z_2 - 1}{2} + 0.652}{2|\ell(z_2)^\frac{1}{2} + \ell(\ell(z_2))^\frac{1}{4}|}.
\end{align*}
This fraction is decreasing as $z_2$ increases, and, when $n \ge 3350194786$, $z_2 \ge 250$, so $\ell(z_2) \ge 58006$ and $\ell(\ell(z_2)) \ge 3350194786$.  The result follows.
\end{proof}

\begin{lem}\label{lem:z3tdiff}
Let $t = \zeta(\lceil n^\frac{1}{4} \rceil)$. If $n \ge 3350194786$, then $\frac{1}{4}|z_3 - t| \le 0.25$.
\end{lem}
\begin{proof}
The assumption on $n$ guarantees that $z_3 \ge 18$. As always, $z_1 \ge \ell(z_2)$ and $z_2 \ge \ell(z_3)$, so
\begin{align*}
z_1 &\ge \ell(z_3)^2 - (z_3 + 1)\ell(z_3) + b(z_3) + 3 \\
&> \left(\ell(z_3) - \frac{z_3+1}{2} \right)^2 > \left(b(z_3 - 1) - \frac{z_3 + 1}{2} \right)^2,
\end{align*}
where the last inequality is true by Lemma \ref{lem:bandl}(2). Applying this and Lemma \ref{lem:rootzrootrootndiff}, we see that
\begin{align*}
n^\frac{1}{4} &= z_1^\frac{1}{2} - \left(z_1^\frac{1}{2} - n^\frac{1}{4} \right)\\ 
&> b(z_3 - 1) - \frac{z_3 + 1}{2} - 1
> b(z_3 - 1) - (z_3 - 1) + 3 
= \ell(z_3 - 1),
\end{align*}
where the second inequality is valid because $z_3 \ge 18$. For an upper bound on $\lceil n^\frac{1}{4} \rceil$, we use Lemma \ref{lem:betterzbounds} to get
\begin{equation*}
\lceil n^\frac{1}{4} \rceil \le \lceil z_1^\frac{1}{2} \rceil \le z_2 < \ell(z_3 + 1). 
\end{equation*}
Thus, $\ell(z_3 - 1) \le \lceil n^\frac{1}{4} \rceil < \ell(z_3 + 1)$, which means that $z_3 - 1 \le t \le z_3$, as required.
\end{proof}

\begin{lem}\label{lem:rdiff}
Let $x \ge 4$ be an integer, let $N = \lfloor \log_2 \log_5 x \rfloor + 1$, and let $y$ be a real number such that $x-1 < y < x$.
\begin{enumerate}[(1)]
\item If $x \ne 5^{2^d}$ for any $d \in \N$, then $r(x) - r(y) < \dfrac{2^N - 1}{2^N \sqrt{y}}$.
\item If $x = 5^{2^d}$ for some $d \in \N$, then $r(x) - r(y) < \dfrac{2^{d} - 1}{2^{d}\sqrt{y}} + \dfrac{\sqrt{5} - 1}{2^{d}}$.
\end{enumerate}
\end{lem}
\begin{proof}
Note that $a^\frac{1}{2^k} - c^\frac{1}{2^k} < a^\frac{1}{2} - c^\frac{1}{2}$ whenever $a, c, k > 1$. If $x \ne 5^{2^d}$ for all $d \in \N$, then $\lfloor \log_2 \log_5 y \rfloor = \lfloor \log_2 \log_5 x \rfloor$, so \begin{align*}
r(x) - r(y) &= \sum_{j=1}^N \frac{x^\frac{1}{2^j} - y^\frac{1}{2^j}}{2^{j-1}}
\le \left(x^\frac{1}{2} - y^\frac{1}{2}\right) \sum_{j=1}^N \frac{1}{2^{j-1}}\\
&= \frac{\left(2 - \frac{1}{2^{N-1}}\right)(x-y)}{\left(x^\frac{1}{2} + y^\frac{1}{2} \right)} 
< \frac{2^N - 1}{2^{N}\sqrt{y}}.
\end{align*}
This proves (1). For (2), we have $N=d+1$ and $\lfloor \log_2 \log_5 y \rfloor = \lfloor \log_2 \log_5 x \rfloor - 1$. Using the above work, we see that
\begin{align*}
r(x) - r(y) &= \sum_{j=1}^{N} \frac{1}{2^{j-1}}x^\frac{1}{2^j} - \left(1 - \frac{1}{2^{N-1}}\right)
- \sum_{j=1}^{N-1} \frac{1}{2^{j-1}}y^\frac{1}{2^j} + \left(1 - \frac{1}{2^{N-2}}\right)\\
&= \sum_{j=1}^{N-1} \frac{x^\frac{1}{2^j} - y^\frac{1}{2^j}}{2^{j-1}} + \frac{1}{2^{N-1}} - \frac{1}{2^{N-2}} + \frac{1}{2^{N-1}}x^\frac{1}{2^{N}}\\
&< \frac{2^{N-1} - 1}{2^{N-1}\sqrt{y}} + \frac{\sqrt{5} - 1}{2^{N-1}},
\end{align*} 
which is the desired bound.
\end{proof}

We now have what we need to prove that the difference between $\zeta(n)$ and $r(n)$ is bounded absolutely.

\begin{proof}[Proof of Theorem \ref{thm:restimate}]
We proceed by induction.  First, to show that the result holds for $n < 3350194786$, we note that there are only four possibilities for  $N = \lfloor \log_2 \log_5 n \rfloor + 1$ in these cases: namely, $N = 1$ when $n < 25$; $N = 2$ when $25 \le n < 625$; $N = 3$ when $625 \le n < 5^8 = 390625$; and $N = 4$ for $5^8 \le n < 3350194786$.  When $n < 3350194786$, $z_1 \le 58005$, and, with the exceptions of when $n = 5^{2^d}$ for $1 \le d \le 3$, $r(n)$ is an increasing function for each fixed $z_1$.  This means, for nearly every $z_1 \le 58005$, one only needs to check the extreme possibilities $n = \ell(z_1)$ and $n = \ell(z_1+1) - 1$, which is approximately $116000$ cases (as opposed to more than $3$ billion). We have verified these cases computationally, so the theorem indeed holds for $n < 3350194786$; see Appendix \ref{app:4.6}.

Assume now for some fixed $n \ge 3350194786$ that the result holds for all integers at least $3$ and less than $n$. It suffices to provide bounds for each summand in \eqref{eq:zrdiff}, which can be accomplished via Lemmas \ref{lem:zrootndiff}, \ref{lem:rootzrootrootndiff}, \ref{lem:z3tdiff}, the inductive hypothesis, and Lemma \ref{lem:rdiff}. Let $t = \zeta(\lceil n^\frac{1}{4} \rceil)$.  By the inductive hypothesis, we have $\tfrac{1}{4}|t - r(\lceil n^\frac{1}{4} \rceil)| < 0.5$.  Since $n \ge 3350194786$, we have $z_1 \ge 58006$ and $z_2 \ge 250$. If $\lceil n^\frac{1}{4} \rceil \ne 5^{2^d}$ for any $d \in \N$, then we can bound $|r(\lceil n^\frac{1}{4} \rceil) - r(n^\frac{1}{4})|$ by using Lemma \ref{lem:rdiff} with $N=2$ and $y= 3350194786^\frac{1}{4}$. This gives
\begin{equation*}
|\zeta(n) - r(n)| < 0.652 + 0.4091 + 0.13 + 0.25 + 0.5 + 0.0121 = 1.9532.
\end{equation*}
On the other hand, if $\lceil n^\frac{1}{4} \rceil = 5^{2^d}$ for some $d$, then $d \ge 2$ and $n \ge 624^4+1$. By Lemma \ref{lem:betterzbounds}, $z_1 > \sqrt{n}$ and $z_2 > \sqrt{z_1}$, so $z_1 \ge 624^2+1$ and $z_2 \ge 625$. In this case, applying the same lemmas yields
\begin{equation*}
|\zeta(n) - r(n)| < 0.640 + 0.380 + 0.13 + 0.25 + 0.5 + 0.085 = 1.985,
\end{equation*}
and, therefore, the result holds for all $n \ge 3$.  
\end{proof} 

In practice, the biggest difference we have seen between $\zeta(n)$ and $r(n)$ is about $1.45175$.  Indeed, very often the terms we bounded will be much smaller; for example, we see that
\begin{equation*}
\lim_{n \to \infty} \left|r(\lfloor n^\frac{1}{2}\rfloor) - r(n^\frac{1}{2})\right| = 0,
\end{equation*}
and, based on the proof of Lemma \ref{lem:zrootndiff}, we have
\begin{equation*}
\limsup_{n \to \infty} \left|z_1 - \left(n^\frac{1}{2} + \frac{z_2 - 1}{2} \right) \right| = \frac{5}{8}, \quad  \limsup_{n \to \infty} \frac{1}{2}\left|z_2 - \left(z_1^\frac{1}{2} + \frac{z_3 - 1}{2} \right) \right| = \frac{5}{16}
\end{equation*}
whereas these quantities should also be much smaller than this upper bound infinitely often.  On the other hand, even if we were to use the triangle inequality to expand \eqref{eq:zrdiff} out indefinitely to more and more summands, the above limits superior mean that it is likely impossible to bound $|\zeta(n) - r(n)|$ by anything less than $5/8 + 5/16 + 5/32 + \cdots = (5/8)/(1 - 1/2) = 1.25$ for large $n$.  

With Theorem \ref{thm:restimate} in hand, we can establish a corresponding approximation for $b(n)$. This is the content of Theorem \ref{thm:bestimate}, which we restate for convenience.

\begin{customthm}{\ref{thm:bestimate}}
Define $N := \lfloor \log_2 \log_5 n \rfloor + 1,$ and define
\begin{equation*}
g(n) := n^2 - \left(\sum_{j=1}^N \frac{n^{1 + \frac{1}{2^j}}}{2^{j-1}} \right) + \left(2 - \frac{1}{2^{N-1}}\right)n.
\end{equation*}
For all $n \ge 3$, $|b(n) - g(n)| < 2n$.
\end{customthm}
\begin{proof}
First, note that $g(n) = n^2-r(n)\cdot n + n$. The theorem can be verified computationally for $3 \le n \le 8888$. For $n \ge 8889$, we use Theorems \ref{thm:Bigl(z)} and \ref{thm:restimate} and Lemma \ref{lem:betterzbounds}. We have
\begin{align*}
|b(n) - g(n)| &= \left|\left(n^2 - z_1n + b(z_1) \right) - \left(n^2 - r(n)\cdot n + n \right)\right|\\
&\le n|z_1 - r(n)| + |b(z_1) - n|\\
&< 1.985n + (z_1 - 3)\\
&< 1.985n + \sqrt{2n}\\
&< 2n,
\end{align*}
as desired.
\end{proof}

Theorem \ref{thm:bestimate} provides an upper bound for the difference between $b(n)$ and $g(n)$, but, as noted previously, $r(n)$ will often be a much better estimate for $\zeta(n)$. Hence, $g(n)$ will often be more accurate than the bound in Theorem \ref{thm:bestimate} indicates.

\section{Estimating $|W(n)|$}\label{sect:Westimates}

We close the paper by using the theory we have built up so far to provide an estimate for $|W(n)|$.  We begin with a technical result that strengthens Lemma \ref{lem:betterzbounds}.

\begin{lem}\label{lem:betterzrootndiff}\mbox{}
\begin{itemize} 
\item[(1)] For all $n \ge 3$, $z_1 - \lfloor \sqrt{n} \rfloor < n^\frac{1}{4}$.  In particular, $z_1 < n^\frac{1}{2} + n^\frac{1}{4} = \left(1 + \frac{1}{n^{1/4}} \right)n^\frac{1}{2}$.
\item[(2)] For all $n \ge 9$, $n^2 - b(n) < \left(1 + \frac{1}{n^{1/4}} \right)n^\frac{3}{2}$.
\end{itemize} 
 \end{lem}

\begin{proof}
Part (1) can be verified by inspection for $3 \le n \le 6560$; see Appendix \ref{app:5.1}. So, assume that $n \ge 6561$; then, $N= \lfloor \log_2 \log_5 n \rfloor + 1 \ge 3$, $\log_2 \log_5 n < n^\frac{1}{8}$, and $2.25 < \frac{1}{4}n^\frac{1}{4}$. By Theorem \ref{thm:restimate}, we have $z_1 < r(n) + 2$, so 
 \begin{align*} 
 z_1 - \lfloor \sqrt{n} \rfloor &< \sum_{j = 2}^N \frac{n^{\frac{1}{2^j}}}{2^{j-1}} + \frac{1}{2^{N-1}}+ 2\\
 &< \frac{1}{2}n^\frac{1}{4} + (N - 2)\cdot \frac{1}{4}n^\frac{1}{8} + 2.25\\
 &< \frac{1}{2}n^\frac{1}{4} + \log_2\log_5(n)\cdot \frac{1}{4}n^\frac{1}{8} + 2.25\\
 &< n^\frac{1}{4}.
 \end{align*}
Part (2) is a consequence of Part (1), since $n^2 - b(n) = z_1n - b(z_1) < z_1n.$
\end{proof}

Next, we provide an estimate for $|W(n)|$ that is recursive in nature.

\begin{Def}\label{def:h}
For each $n \in \N$, $n \ge 3$, we define 
\begin{equation*}
h(n) := b(n) - (n-1) + \sum_{k = 1}^{z_1} |W(k)|.
\end{equation*}
\end{Def}

\begin{lem}
\label{lem:hdiff}
 For $n \ge 25$, $||W(n)| - h(n)| < 3n.$
\end{lem}

\begin{proof}
The result follows by inspection when $25 \le n \le 388$; see Appendix \ref{app:5.3}.  Thus, we may assume that $n \ge 389$.  Note that by Corollary \ref{cor:secondrec} any integer $m \in W(n)$ that is larger than $b(n)$ and less than $n^2$ can be realized by adjoining $n - d$ mother vertices to a transitive digraph on $d$ vertices, where $1 \le d \le z_1$. Hence, $h(n)$ provides an upper bound on $|W(n)|$, and the difference between $h(n)$ and $|W(n)|$ comes from the integers in the interval $[n(n-z_1) + (z_1-1), n(n-z_1) + b(z_1)]$, which are counted twice; the integers $m \in W(j) \cap W(k)$, where $j \neq k$; and $\{n^2\}$.  We will bound the number of such integers.

Taken together, the integers in the interval $[n(n-z_1) + (z_1-1), n(n-z_1) + b(z_1)]$ and the single integer $\{n^2\}$ account for $b(z_1) - z_1 + 2 = \ell(z_1) - 1 < n$ such integers.      
 
For each $1 \le d \le z_1$, let $I_d = [n(n-d) + d, n(n-d) +d^2]$. A straightforward calculation shows that, when $d \le \lfloor \sqrt{n} \rfloor$, 
\begin{equation*}
n(n - d) + d^2 < n(n-(d-1)) + (d - 1),
\end{equation*}
and so $I_d \cap I_{d-1} = \varnothing$.  On the other hand, when $d \le z_1$, we have $n \ge \ell(z_1) > b(d) - (d-1)$, which implies 
\begin{equation*}
n(n-d) + b(d) < n(n - (d-1)) + (d-1).
\end{equation*}
This means the only overlap between $I_d$ and $I_{d-1}$ comes from digraphs on $d$ vertices with weights in $[b(d)+1, d^2]$.  Using Lemma \ref{lem:betterzrootndiff}(2), the number of integers in the overlap is at most   
\[ d^2 - b(d) < \left(1 + \frac{1}{d^{1/4}} \right)d^\frac{3}{2} \le \left(1 + \frac{1}{z_1^{1/4}}\right) z_1^\frac{3}{2}.\]
Applying Lemmas \ref{lem:betterzrootndiff}(1) and \ref{lem:betterzbounds} and putting this all together, we see that the difference between $h(n)$ and $|W(n)|$ is bounded by 
\begin{equation*} 
n + \left(z_1 - \left\lfloor \sqrt{n} \right\rfloor \right)\left(1 + \frac{1}{z_1^{1/4}}\right) z_1^\frac{3}{2} <  n + n^\frac{1}{4}\cdot \left(1 + \frac{1}{n^{1/8}}\right) \left( n^\frac{1}{2} + n^\frac{1}{4}\right)^\frac{3}{2} < 3n,
\end{equation*}
since $n \ge 389$, which completes the proof.
\end{proof}

Our goal now is to find an explicit estimate that differs from $h(n)$ by at most a constant times $n$.  Suppose that $\lastFunc(n)$ is our approximation for $h(n)$.  Since 
\begin{equation}
\label{eq:hexpanded}
h(n) = b(n) - (n-1)  + \sum_{k = 1}^{\lfloor \sqrt{n} \rfloor} |W(k)| + \sum_{\lfloor \sqrt{n} \rfloor + 1}^{z_1} |W(k)|,
\end{equation} 
we can build $\lastFunc(n)$ by estimating each part of \eqref{eq:hexpanded}. Our function $g(n)$ is an approximation for $b(n)$, so $b(n)-(n-1) \approx g(n)-n$. For the remaining two terms, the estimates we will use are 
\[\sum_{k = 1}^{\lfloor \sqrt{n} \rfloor} |W(k)| \approx \int_0^{\sqrt{n}} \lastFunc(x) \, dx, \quad  \sum_{\lfloor \sqrt{n} \rfloor + 1}^{z_1} |W(k)| \approx (r(n) - \sqrt{n})n.\]  Putting this together, we believe that a good estimate for $h(n)$ should satisfy
\begin{equation}\label{eq:omegaidea}
\lastFunc(n) = g(n) - n + \int_0^{\sqrt{n}} \lastFunc(x) \, dx + (r(n) - \sqrt{n})n = n^2 - n^\frac{3}{2} + \int_0^{\sqrt{n}} \lastFunc(x) \, dx. 
\end{equation}

We show how to construct such a function $\lastFunc(n)$ as an infinite series. Suppose that $\lastFunc(n) = \sum_{j = 0}^\infty c_j n^{1 + \frac{1}{2^j}}$, where each $c_j$ is a constant. Assuming  that $\lastFunc(0) = 0$ and that $\lastFunc(n)$ admits an interchange of summations and integration, we have
\begin{equation*}
\int_0^{\sqrt{n}} \lastFunc(x) \, dx = \sum_{j = 0}^{\infty} \frac{c_j 2^j}{2^{j+1} + 1}n^{1 + \frac{1}{2^{j+1}}},
\end{equation*}
Comparing the left-hand and right-hand sides of \eqref{eq:omegaidea} gives
\begin{equation*}
c_0 n^2 + c_1 n^\frac{3}{2} + \sum_{j=1}^{\infty} c_{j+1} n^{1 + \frac{1}{2^{j+1}}}= n^2 + (\tfrac{1}{3}c_0-1)n^\frac{3}{2} + \sum_{j = 1}^{\infty} \frac{c_j 2^j}{2^{j+1} + 1}n^{1 + \frac{1}{2^{j+1}}}.
\end{equation*}
Equating coefficients, this implies that $c_0 = 1$, $c_1 = -2/3$, and, for all $j \ge 1$,
\[c_{j+1} = \frac{c_j 2^j}{2^{j+1} + 1}. \]
It follows by induction that, for all $k \ge 1$, 
\[c_k = \frac{-2^k}{\prod_{i=1}^k (2^i + 1)}.\]
We now define our function $\lastFunc(n)$.

\begin{Def}\label{def:omega}
For each real number $x \ge 0$, define
\begin{equation*}
\lastFunc(x) := x^2 - \sum_{k = 1}^{\infty} \left(\frac{2^k}{\prod_{i=1}^k (2^i + 1)}\right) x^{1 + \frac{1}{2^k}}.
\end{equation*}
When $x \ge 3$, let $N:=\lfloor \log_2 \log_5 x \rfloor + 1$, and define the truncated summation
\begin{equation*}
\oom(x) := x^2 - \sum_{k = 1}^{N} \left(\frac{2^k}{\prod_{i=1}^k (2^i + 1)}\right) x^{1 + \frac{1}{2^k}}.
\end{equation*}
\end{Def}

\begin{lem}\label{lem:convergence}\mbox{}
\begin{enumerate}[(1)]
\item For each $x \ge 0$, $\lastFunc(x)$ is well-defined.
\item $\lastFunc(x)$ satisfies \eqref{eq:omegaidea}.
\item For $x \ge 25$, $|\oom(x) - \lastFunc(x)| < \frac{1}{3}x$.
\end{enumerate}
\end{lem}
\begin{proof}
For each $k$, let $c_k = -2^k/\prod_{i=1}^k (2^i + 1)$. Clearly, $\lastFunc(0)=0$, and when $x > 0$, the ratio of consecutive terms in $\sum_{k = 1}^{\infty} c_k x^{1 + \frac{1}{2^k}}$ is equal to
\begin{equation*}
\frac{2}{(2^{k+1}+1)x^\frac{1}{2^{k+1}}},
\end{equation*}
which converges to 0 as $k \to \infty$. Hence, $\sum_{k = 1}^{\infty} c_k x^{1 + \frac{1}{2^k}}$ is absolutely convergent for all $x \ge 0$, and $\lastFunc(x)$ is well defined. It is now clear from the discussion prior to Definition \ref{def:omega} that $\lastFunc(n) = n^2 - n^\frac{3}{2} + \int_{0}^{\sqrt{n}} \lastFunc(x) \, dx$.

For (3), a straightforward induction shows that $\sum_{k=1}^m (-c_k) = 1 - 1/\prod_{i=1}^m(2^i+1)$, and so $\sum_{k=1}^\infty (-c_k) = 1$. By construction, $x$ and $N$ satisfy $x < 5^{2^N}$, and $N+1 \ge 3$ because $x \ge 25$. Hence,
\begin{equation*}
\oom(x) - \lastFunc(x) = \sum_{k=N+1}^\infty -c_k x^{1+\frac{1}{2^k}} < \sum_{k=3}^\infty -c_k(5x) < \frac{1}{3}x,
\end{equation*}
as required.
\end{proof}

As we will show, one can use either $\lastFunc(n)$ or $\oom(n)$ to estimate $|W(n)|$. The truncated series $\oom(n)$ is more amenable to calculation, but $\lastFunc(n)$ is easier to work with in proofs. In the lemmas below, we will bound the difference between each term in \eqref{eq:hexpanded} and its respective term in \eqref{eq:omegaidea}. These will later be used to bound $||W(n)| - \oom(n)|$. The bounds we establish are not optimal; we are satisfied as long as our final bound for $||W(n)| - \oom(n)|$ is a constant multiple of $n$.

\begin{lem}
\label{lem:omega2}
 For $n \ge 3$,
 \[ \left|\sum_{k=1}^{\lfloor \sqrt{n} \rfloor}|W(k)| - \int_0^{\sqrt{n}} \lastFunc(x) \, dx\right| < \left(\sum_{k=1}^{\lfloor \sqrt{n} \rfloor}\left||W(k)| - \lastFunc(k) \right|\right) + 2n.\]
\end{lem}

\begin{proof}
First, by the triangle inequality,
\begin{align*}
\left|\sum_{k=1}^{\lfloor \sqrt{n} \rfloor}|W(k)| - \int_0^{\sqrt{n}} \lastFunc(x) \, dx\right| 
\le& \left(\sum_{k=1}^{\lfloor \sqrt{n} \rfloor}\left||W(k)| - \lastFunc(k) \right|\right)\\
   &+ \left|\sum_{k=1}^{\lfloor \sqrt{n} \rfloor} \lastFunc(k) - \int_0^{\lfloor\sqrt{n}\rfloor} \lastFunc(x)\, dx \right| + \int_{\lfloor \sqrt{n}\rfloor}^{\sqrt{n}} \lastFunc(x) \, dx.
\end{align*}
We will show that each of the last two summands on the right-hand side can be bounded by $n$.

One may verify that $\lastFunc(x)$ is negative for $0 < x < 1$, and $\big|\int_0^1 \lastFunc(x) \, dx\big| \le 1$. Moreover, $\lastFunc(1)=0$ and $\lastFunc(x)$ is increasing for $x > 1$, so
\begin{equation*}
\sum_{k=1}^{\lfloor \sqrt{n} \rfloor-1} \lastFunc(k) \le \int_{1}^{\lfloor \sqrt{n} \rfloor} \lastFunc(x) \, dx \le \sum_{k=1}^{\lfloor \sqrt{n} \rfloor} \lastFunc(k).
\end{equation*}
It follows that
\begin{equation*}
\left| \sum_{k=1}^{\lfloor \sqrt{n} \rfloor} \lastFunc(k) - \int_{0}^{\lfloor \sqrt{n} \rfloor} \lastFunc(x) \, dx \right| < \lastFunc(\lfloor \sqrt{n} \rfloor) +1 \le n - 1 + 1 = n.
\end{equation*}
Finally, since $\lastFunc(x)$ is increasing for on $\lfloor \sqrt{n} \rfloor \le x \le \sqrt{n}$, we get
\begin{equation*}
\int_{\lfloor \sqrt{n}\rfloor}^{\sqrt{n}} \lastFunc(x) \, dx \le \lastFunc(\sqrt{n})\left( \sqrt{n} - \lfloor \sqrt{n}\rfloor\right) < \lastFunc(\sqrt{n}) < n,
\end{equation*}
which implies the stated result.
\end{proof}

\begin{lem}
\label{lem:omega3}
 For $n \ge 389$, we have
 \[\left|\sum_{k = \lfloor \sqrt{n} \rfloor + 1}^{z_1}|W(k)| - (r(n) - \sqrt{n})n \right| < 7n. \]
\end{lem}

\begin{proof}
For readability, let $m= \lfloor \sqrt{n} \rfloor+1$. By the triangle inequality, we have
\begin{align*}
\left|\sum_{k = m}^{z_1}|W(k)| - (r(n) - \sqrt{n})n \right| 
\le &\left|\sum_{k = m}^{z_1}\left(|W(k)| - b(k)\right) \right| 
+ \left| \sum_{k= m}^{z_1} b(k) - (z_1 - \lfloor \sqrt{n} \rfloor)n\right|\\ 
&+ \big|(z_1 - \lfloor \sqrt{n}\rfloor)n - (r(n) - \sqrt{n})n\big|. 
\end{align*}
Since by Lemma \ref{lem:betterzrootndiff}(2) we have
\[||W(k)| - b(k)| < k^2 - b(k) < \left(1 + \frac{1}{k^{1/4}} \right)k^\frac{3}{2} \le \left(1 + \frac{1}{z_1^{1/4}} \right)z_1^\frac{3}{2},\] 
we can use Lemma \ref{lem:betterzrootndiff}(1) to bound the first summand on the right-hand side by
\begin{equation*} 
\left|\sum_{k =m}^{z_1}\left(|W(k)| - b(k)\right) \right| < (z_1 - \lfloor \sqrt{n} \rfloor)(z_1^2 - b(z_1)) < n^\frac{1}{4} \cdot \left(1 + \frac{1}{z_1^{1/4}} \right)z_1^\frac{3}{2} < 2n 
\end{equation*}
in a similar fashion to the bound obtained in the proof of Lemma \ref{lem:hdiff}.
 
Next, for the second summand notice that $b(z_1 - 1) < \ell(z_1) \le n$ and, by Lemma \ref{lem:betterzrootndiff}(2), 
\begin{equation*}
b(m) > m^2 - m^\frac{3}{2} - m^\frac{5}{4} > n - m^\frac{3}{2} - m^\frac{5}{4}.
\end{equation*}
Hence, $n - m^\frac{3}{2} - m^\frac{5}{4} < b(k) < n$ for all $m \le k \le z_1-1$, and so the second summand is bounded by 
\[\left|b(z_1) - n\right| + ((z_1 -1) - m + 1)(n - (n - m^\frac{3}{2} - m^\frac{5}{4})),\]
which is in turn bounded by 
\begin{equation*}
(z_1 - 3) + n^\frac{1}{4}(m^\frac{3}{2} + m^\frac{5}{4}) < (n^\frac{1}{2} + n^\frac{1}{4}) + n^\frac{1}{4}\left((\sqrt{n}+1)^\frac{3}{2} + (\sqrt{n}+1)^\frac{5}{4} \right) < 2n.
\end{equation*}
Finally, we use Theorem \ref{thm:restimate} to bound the last summand on the right-hand side by
\begin{equation*}
\big|(z_1 - \lfloor \sqrt{n}\rfloor)n - (r(n) - \sqrt{n})n\big| \le n\big|z_1 -r(n)\big| + n\big|\sqrt{n} - \lfloor \sqrt{n} \rfloor\big| < 3n.
\end{equation*}
The result follows.
\end{proof}

We are now ready to prove Theorem \ref{thm:Westimate}.

\begin{proof}[Proof of Theorem \ref{thm:Westimate}]
By Lemma \ref{lem:convergence}(3), it suffices to show that $||W(n)-\lastFunc(n)| < \frac{89}{3}n$ for all $n \ge 3$. This claim holds by inspection for $3 \le n \le 509$; see Appendix \ref{app:5.3}.  Now assume that $n \ge 510$ and the result holds for all natural numbers between $3$ and $n-1$. From Lemma \ref{lem:hdiff}, we obtain
\begin{equation*}
||W(n)| - \lastFunc(n)| \le ||W(n)|-h(n)| + |h(n)-\lastFunc(n)| < 3n + |h(n)-\lastFunc(n)|.
\end{equation*}
To bound $|h(n)-\lastFunc(n)|$, we use \eqref{eq:hexpanded}, \eqref{eq:omegaidea}, Theorem \ref{thm:bestimate}, and Lemmas \ref{lem:omega3} and \ref{lem:omega2} to get
\begin{align*}
|h(n) - \lastFunc(n)| &\le \big|b(n)-(n-1)-g(n)+n\big| + \left|\sum_{k=1}^{\lfloor \sqrt{n} \rfloor}|W(k)| - \int_0^{\sqrt{n}} \lastFunc(x) \, dx\right| \\
& \quad + \left|\sum_{k = \lfloor \sqrt{n} \rfloor + 1}^{z_1}|W(k)| - (r(n) - \sqrt{n} - 1)n \right|\\
&< 1+11n + \sum_{k=1}^{\lfloor \sqrt{n} \rfloor}||W(k)| - \lastFunc(k) |.
\end{align*}
Applying the inductive hypothesis to $||W(k)| - \lastFunc(k) |$ and recalling that $n \ge 510$ yields
\begin{equation*}
\sum_{k=1}^{\lfloor \sqrt{n} \rfloor}\left||W(k)| - \lastFunc(k) \right| < 30\sum_{k=1}^{\lfloor \sqrt{n}\rfloor} k = 15\lfloor \sqrt{n} \rfloor\left(\lfloor \sqrt{n} \rfloor+ 1 \right) < 15(n+\sqrt{n}) < \frac{47}{3}n - 1.
\end{equation*}
Combining all the bounds shows that $||W(n)|-\lastFunc(n)| < \frac{89}{3}n$, and the theorem follows.
\end{proof}


\appendix
\section{Data and Mathematica code}
\label{app:A}
\subsection{Data for Proposition \ref{prop:3/4b}}
\label{app:2.1}

If we have already determined the weight sets $W(1), W(2), \ldots, W(n-1)$, then Lemma \ref{Subsets of W(n)} can be used to construct a subset of $W(n)$. It is easily seen that $W(1)=\{1\}$, $W(2)=[2,4]$, and $W(3)=[3,7]\cup\{9\}$. Using these as a starting point, we can produce subsets of $W(n)$ for $n \ge 4$, and these subsets can be used to find lower bounds on $b(n)$. Table \ref{b(n) table} summarizes these results for $1 \le n \le 18$ and lists the corresponding values of $(3/4)n^2$.

\begin{table}[H]
\begin{minipage}{0.5\linewidth}\centering
{\footnotesize
\begin{tabular}{ccc}
$n$ & $b(n)$ lower bound & $(3/4)n^2$\\
\hline
1 & 1 & 0.75\\
2 & 4 & 3\\
3 & 7 & 6.75\\
4 & 13 & 12\\
5 & 19 & 18.75\\
6 & 28 & 27\\
7 & 35 & 36.75\\
8 & 52 & 48\\
9 & 61 & 60.75
\end{tabular}
}\end{minipage}%
\begin{minipage}{0.5\linewidth}\centering
{\footnotesize
\begin{tabular}{ccc}
$n$ & $b(n)$ lower bound & $(3/4)n^2$\\
\hline
10 & 77 & 75\\
11 & 95 & 90.75\\
12 & 109 & 108\\
13 & 130 & 126.75\\
14 & 153 & 147\\
15 & 178 & 168.75\\
16 & 205 & 192\\
17 & 223 & 216.75\\
18 & 253 & 243
\end{tabular}
}\end{minipage}
\caption{Lower bounds for $b(n)$, using Lem.\ \ref{Subsets of W(n)}}
\label{b(n) table}
\end{table}

\subsection{Data for Theorem \ref{thm:Bigl(z)}}
\label{app:1.4}
Define functions and variables as follows.
\begin{itemize}
\item For each $n \ge 1$, $b(n)$ is the largest positive integer such that $[n,b(n)] \subseteq W(n)$.
\item For each $z \ge 1$, $\ell(z) = b(z)-z+3$.
\item For each $n \ge 3$, $\zeta(n)$ is the unique positive integer such that $\ell(\zeta(n)) \le n < \ell(\zeta(n)+1)$. To save space we often let $z := \zeta(n)$ or $z_1 := \zeta(n)$.
\end{itemize}
Note that $b(n)$ can be found as soon as we know $W(n)$; and $W(n)$ can be computed by using the recursion proved in Corollary \ref{cor:firstrec}.

\begin{table}[H]
\begin{minipage}{0.5\linewidth}\centering
{\footnotesize
\begin{tabular}{ccccc}
$n$ & $b(n)$ & $\ell(n)$ & $z:=\zeta(n)$ & $n^2-zn+b(z)$\\
\hline
1 & 1 & 3 & $-$ & $-$\\
2 & 4 & 5 & $-$ & $-$\\
3 & 7 & 7 & 1 & 7 \\
4 & 13 & 12 & 1 & 13 \\
5 & 19 & 17 & 2 & 19 \\
6 & 28 & 25 & 2 & 28 \\
7 & 35 & 31 & 3 & 35 \\
8 & 52 & 47 & 3 & 47 \\
9 & 61 & 55 & 3 & 61 \\
10 & 77 & 70 & 3 & 77 \\
11 & 95 & 87 & 3 & 95 \\
12 & 109 & 100 & 4 & 109 
\end{tabular}
}\end{minipage}%
\begin{minipage}{0.5\linewidth}\centering
{\footnotesize
\begin{tabular}{ccccc}
$n$ & $b(n)$ & $\ell(n)$ & $z:=\zeta(n)$ & $n^2-zn+b(z)$\\
\hline
13 & 130 & 120 & 4 & 130 \\
14 & 153 & 142 & 4 & 153 \\
15 & 178 & 166 & 4 & 178 \\
16 & 205 & 192 & 4 & 205 \\
17 & 223 & 209 & 5 & 223 \\
18 & 253 & 238 & 5 & 253 \\
19 & 285 & 269 & 5 & 285 \\
20 & 319 & 302 & 5 & 319 \\
21 & 355 & 337 & 5 & 355 \\
22 & 393 & 374 & 5 & 393 \\
23 & 433 & 413 & 5 & 433 \\
24 & 475 & 454 & 5 & 475
\end{tabular}
}\end{minipage}
\caption{Values of $b(n)$, $\ell(n)$, and $\zeta(n)$ for $1 \le n \le 24$}
\label{Small n table}
\end{table}

\subsubsection{Algorithm and Code for $b(n)$}
\label{app:bcode}
Having established Theorem \ref{thm:Bigl(z)}, we can compute $b(n)$ without first knowing $W(n)$. Assume that we know the values of $b(k)$ for all $1 \le k \le n$. The value of $b(n+1)$ can be found by implementing the following algorithm:
\begin{enumerate}
\item Compute $\ell(k)$ for each $1 \le k \le n$.
\item Find $x:=\zeta(n+1)$, which is the integer such that $\ell(x) \le n+1 < \ell(x+1)$. Note that $1 \le x \le n$.
\item Compute $b(n+1) = (n+1)^2 - x(n+1)+b(x)$.
\end{enumerate}

The code below will generated $b(n)$, $\ell(n)$, and $\zeta(n)$ in Mathematica. First, execute this snippet to initialize the variables:
\begin{verbatim}
Clear[n, b, z, L];
bTable = {1, 4, 7, 13, 19, 28, 35, 52};
L[z_] := bTable[[z]] - z + 3;
b[n_] := n^2 - z[n]*n + bTable[[z[n]]];
LTable = Table[L[z], {z, 1, Length[bTable]}];
FindLval[n_] := Max[Select[LTable, # <= n &]];
z[n_] := Position[LTable, FindLval[n]][[1]][[1]];
zTable = 
   Prepend[Prepend[Table[z[n], {n, 3, Length[LTable]}], 0], 0];
\end{verbatim}
Then, execute the following commands to calculate $b(n)$, $\ell(n)$, and $\zeta(n)$ for all $n$ less than or equal to the value of \texttt{bound}.
\begin{verbatim}
bound = 1000;
For[i = Length[bTable] + 1, i <= bound, i++, 
 zTable = Append[zTable, z[i]]; bTable = Append[bTable, b[i]]; 
 LTable = Append[LTable, L[i]]]
\end{verbatim}
The \texttt{1000} above should be replaced with the desired value of $n$. The values of $b$, $\ell$, and $\zeta$ are stored, respectively, in \texttt{bTable}, \texttt{LTable}, and \texttt{zTable}.

\subsection{Data for Lemma \ref{lem:betterzbounds}}
\label{app:4.3}

\begin{table}[H]
\begin{minipage}{0.33\linewidth}\centering
{\footnotesize
\begin{tabular}{cccc}
$n$ & $\sqrt{n}$ & $\zeta(n)$ & $\sqrt{2n}$ \\
\hline
 11 & 3.32 & 3 & 4.690 \\
 12 & 3.46 & 4 & 4.899 \\
 13 & 3.61 & 4 & 5.099 \\
 14 & 3.74 & 4 & 5.292 \\
 15 & 3.87 & 4 & 5.477 \\
 16 & 4 & 4 & 5.657 \\
 17 & 4.12 & 5 & 5.831 \\
 18 & 4.24 & 5 & 6 \\
 19 & 4.36 & 5 & 6.164 \\
 20 & 4.47 & 5 & 6.325 \\
 21 & 4.58 & 5 & 6.481 \\
 22 & 4.69 & 5 & 6.633 \\
 23 & 4.80 & 5 & 6.782 \\
 24 & 4.90 & 5 & 6.928 \\
 25 & 5 & 6 & 7.071 \\
 26 & 5.10 & 6 & 7.211 \\
 27 & 5.20 & 6 & 7.348 \\
 28 & 5.29 & 6 & 7.483 \\
 29 & 5.39 & 6 & 7.616 \\
 30 & 5.48 & 6 & 7.746 \\
 31 & 5.57 & 7 & 7.874 \\
 32 & 5.66 & 7 & 8 \\
 33 & 5.74 & 7 & 8.124 \\
 34 & 5.83 & 7 & 8.246 \\
 35 & 5.92 & 7 & 8.367 \\
 36 & 6 & 7 & 8.485 \\
 37 & 6.08 & 7 & 8.602 \\
 38 & 6.16 & 7 & 8.718 \\
 39 & 6.24 & 7 & 8.832 \\
 40 & 6.32 & 7 & 8.944 \\
\end{tabular}
}\end{minipage}%
\begin{minipage}{0.33\linewidth}\centering
{\footnotesize
\begin{tabular}{cccc}
$n$ & $\sqrt{n}$ & $\zeta(n)$ & $\sqrt{2n}$ \\
\hline
 41 & 6.40 & 7 & 9.055 \\
 42 & 6.48 & 7 & 9.165 \\
 43 & 6.56 & 7 & 9.274 \\
 44 & 6.63 & 7 & 9.381 \\
 45 & 6.71 & 7 & 9.487 \\
 46 & 6.78 & 7 & 9.592 \\
 47 & 6.86 & 8 & 9.695 \\
 48 & 6.93 & 8 & 9.798 \\
 49 & 7 & 8 & 9.899 \\
 50 & 7.07 & 8 & 10 \\
 51 & 7.14 & 8 & 10.10 \\
 52 & 7.21 & 8 & 10.20 \\
 53 & 7.28 & 8 & 10.30 \\
 54 & 7.35 & 8 & 10.39 \\
 55 & 7.42 & 9 & 10.49 \\
 56 & 7.48 & 9 & 10.58 \\
 57 & 7.55 & 9 & 10.68 \\
 58 & 7.62 & 9 & 10.77 \\
 59 & 7.68 & 9 & 10.86 \\
 60 & 7.75 & 9 & 10.95 \\
 61 & 7.81 & 9 & 11.05 \\
 62 & 7.87 & 9 & 11.14 \\
 63 & 7.94 & 9 & 11.22 \\
 64 & 8 & 9 & 11.31 \\
 65 & 8.06 & 9 & 11.40 \\
 66 & 8.12 & 9 & 11.49 \\
 67 & 8.19 & 9 & 11.58 \\
 68 & 8.25 & 9 & 11.66 \\
 69 & 8.31 & 9 & 11.75 \\
 70 & 8.37 & 10 & 11.83 
\end{tabular}
}\end{minipage}%
\begin{minipage}{0.33\linewidth}\centering
{\footnotesize
\begin{tabular}{cccc}
$n$ & $\sqrt{n}$ & $\zeta(n)$ & $\sqrt{2n}$ \\
\hline
 71 & 8.43 & 10 & 11.92 \\
 72 & 8.49 & 10 & 12 \\
 73 & 8.54 & 10 & 12.08 \\
 74 & 8.60 & 10 & 12.17 \\
 75 & 8.66 & 10 & 12.25 \\
 76 & 8.72 & 10 & 12.33 \\
 77 & 8.77 & 10 & 12.41 \\
 78 & 8.83 & 10 & 12.49 \\
 79 & 8.89 & 10 & 12.57 \\
 80 & 8.94 & 10 & 12.65 \\
 81 & 9 & 10 & 12.73 \\
 82 & 9.06 & 10 & 12.81 \\
 83 & 9.11 & 10 & 12.88 \\
 84 & 9.17 & 10 & 12.96 \\
 85 & 9.22 & 10 & 13.04 \\
 86 & 9.27 & 10 & 13.11 \\
 87 & 9.33 & 11 & 13.19 \\
 88 & 9.38 & 11 & 13.27 \\
 89 & 9.43 & 11 & 13.34 \\
 90 & 9.49 & 11 & 13.42 \\
 91 & 9.54 & 11 & 13.49 \\
 92 & 9.59 & 11 & 13.56 \\
 93 & 9.64 & 11 & 13.64 \\
 94 & 9.70 & 11 & 13.71 \\
 95 & 9.75 & 11 & 13.78 \\
 96 & 9.80 & 11 & 13.86 \\
 97 & 9.85 & 11 & 13.93 \\
 98 & 9.90 & 11 & 14 \\
 99 & 9.95 & 11 & 14.07 \\
100 & 10 & 12 & 14.14
\end{tabular}
}\end{minipage}
\caption{Values of $\sqrt{n}$, $\zeta(n)$, and $\sqrt{2n}$ for $11 \le n \le 100$}
\label{sqrt table}
\end{table}

\subsection{Data for Lemma \ref{lem:bandl}(1)}
\label{app:4.4(1)}
In Table \ref{b difference table}, $\Delta b$ refers to $b(n+1) - b(n)$.

\begin{table}[H]
\begin{minipage}{0.49\linewidth}\centering
{\footnotesize
\begin{tabular}{ccccc}
$n$ & $b(n)$ & $n+1$ & $\Delta b$ & $2n+1$\\
\hline
 8 & 52 & 9 & 9 & 17 \\
 9 & 61 & 10 & 16 & 19 \\
 10 & 77 & 11 & 18 & 21 \\
 11 & 95 & 12 & 14 & 23 \\
 12 & 109 & 13 & 21 & 25 \\
 13 & 130 & 14 & 23 & 27 \\
 14 & 153 & 15 & 25 & 29 \\
 15 & 178 & 16 & 27 & 31 \\
 16 & 205 & 17 & 18 & 33 \\
 17 & 223 & 18 & 30 & 35 \\
 18 & 253 & 19 & 32 & 37 \\
 19 & 285 & 20 & 34 & 39 \\
\end{tabular}
}\end{minipage}%
\begin{minipage}{0.49\linewidth}\centering
{\footnotesize
\begin{tabular}{ccccc}
$n$ & $b(n)$ & $n+1$ & $\Delta b$ & $2n+1$\\
\hline
 20 & 319 & 21 & 36 & 41 \\
 21 & 355 & 22 & 38 & 43 \\
 22 & 393 & 23 & 40 & 45 \\
 23 & 433 & 24 & 42 & 47 \\
 24 & 475 & 25 & 28 & 49 \\
 25 & 503 & 26 & 45 & 51 \\
 26 & 548 & 27 & 47 & 53 \\
 27 & 595 & 28 & 49 & 55 \\
 28 & 644 & 29 & 51 & 57 \\
 29 & 695 & 30 & 53 & 59 \\
 30 & 748 & 31 & 31 & 61 \\
 31 & 779 & 32 & 56 & 63 \\
\end{tabular}
}\end{minipage}
\caption{Values of $n+1$, $\Delta b := b(n+1)-b(n)$, and $2n+1$ for $8 \le n \le 31$}
\label{b difference table}
\end{table}

\begin{table}[H]
\begin{minipage}{0.49\linewidth}\centering
{\footnotesize
\begin{tabular}{ccccc}
$n$ & $b(n)$ & $n+1$ & $\Delta b$ & $2n+1$\\
\hline
 32 & 835 & 33 & 58 & 65 \\
 33 & 893 & 34 & 60 & 67 \\
 34 & 953 & 35 & 62 & 69 \\
 35 & 1015 & 36 & 64 & 71 \\
 36 & 1079 & 37 & 66 & 73 \\
 37 & 1145 & 38 & 68 & 75 \\
 38 & 1213 & 39 & 70 & 77 \\
 39 & 1283 & 40 & 72 & 79 \\
\end{tabular}
}\end{minipage}%
\begin{minipage}{0.49\linewidth}\centering
{\footnotesize
\begin{tabular}{ccccc}
$n$ & $b(n)$ & $n+1$ & $\Delta b$ & $2n+1$\\
\hline
 40 & 1355 & 41 & 74 & 81 \\
 41 & 1429 & 42 & 76 & 83 \\
 42 & 1505 & 43 & 78 & 85 \\
 43 & 1583 & 44 & 80 & 87 \\
 44 & 1663 & 45 & 82 & 89 \\
 45 & 1745 & 46 & 84 & 91 \\
 46 & 1829 & 47 & 56 & 93 \\
 47 & 1885 & 48 & 87 & 95 \\
\end{tabular}
}\end{minipage}
\caption{Values of $n+1$, $\Delta b := b(n+1)-b(n)$, and $2n+1$ for $32 \le n \le 47$}
\label{b difference table1}
\end{table}

\subsection{Data for Lemma \ref{lem:bandl}(3)}
\label{app:4.4(3)} 
In Tables \ref{b-z1 table1}, \ref{b-z1 table2}, and \ref{b-z1 table3}, $z_1 := \zeta(n)$.

\begin{table}[H]
\begin{minipage}{0.5\linewidth}\centering
{\footnotesize
\begin{tabular}{ccccc}
$n$ & $z_1$ & $b(z_1)$ & $|n-b(z_1)|$ & $z_1-3$\\
\hline
 191 & 15 & 178 & 13 & 12 \\
 192 & 16 & 205 & 13 & 13 \\
 193 & 16 & 205 & 12 & 13 \\
 194 & 16 & 205 & 11 & 13 \\
 195 & 16 & 205 & 10 & 13 \\
 196 & 16 & 205 & 9 & 13 \\
 197 & 16 & 205 & 8 & 13 \\
 198 & 16 & 205 & 7 & 13 \\
 199 & 16 & 205 & 6 & 13 \\
 200 & 16 & 205 & 5 & 13 \\
 201 & 16 & 205 & 4 & 13 \\
 202 & 16 & 205 & 3 & 13 \\
 203 & 16 & 205 & 2 & 13 \\
 204 & 16 & 205 & 1 & 13 \\
 205 & 16 & 205 & 0 & 13 \\
 206 & 16 & 205 & 1 & 13 \\
 207 & 16 & 205 & 2 & 13 \\
 208 & 16 & 205 & 3 & 13 \\
 209 & 17 & 223 & 14 & 14 \\
 210 & 17 & 223 & 13 & 14 \\ 
 211 & 17 & 223 & 12 & 14 \\
 212 & 17 & 223 & 11 & 14 \\
 213 & 17 & 223 & 10 & 14 \\
 214 & 17 & 223 & 9 & 14 \\
 215 & 17 & 223 & 8 & 14 \\
 216 & 17 & 223 & 7 & 14 \\
 217 & 17 & 223 & 6 & 14 \\
 218 & 17 & 223 & 5 & 14 \\
 219 & 17 & 223 & 4 & 14 \\
 220 & 17 & 223 & 3 & 14 \\
 221 & 17 & 223 & 2 & 14 \\
 222 & 17 & 223 & 1 & 14 \\
 223 & 17 & 223 & 0 & 14 \\
 224 & 17 & 223 & 1 & 14 \\
 225 & 17 & 223 & 2 & 14 \\
 226 & 17 & 223 & 3 & 14 \\
 227 & 17 & 223 & 4 & 14 \\
\end{tabular}
}\end{minipage}%
\begin{minipage}{0.5\linewidth}\centering
{\footnotesize
\begin{tabular}{ccccc}
$n$ & $z_1$ & $b(z_1)$ & $|n-b(z_1)|$ & $z_1-3$\\
\hline 
 228 & 17 & 223 & 5 & 14 \\
 229 & 17 & 223 & 6 & 14 \\
 230 & 17 & 223 & 7 & 14 \\
 231 & 17 & 223 & 8 & 14 \\
 232 & 17 & 223 & 9 & 14 \\
 233 & 17 & 223 & 10 & 14 \\
 234 & 17 & 223 & 11 & 14 \\
 235 & 17 & 223 & 12 & 14 \\
 236 & 17 & 223 & 13 & 14 \\
 237 & 17 & 223 & 14 & 14 \\
 238 & 18 & 253 & 15 & 15 \\
 239 & 18 & 253 & 14 & 15 \\
 240 & 18 & 253 & 13 & 15 \\
 241 & 18 & 253 & 12 & 15 \\
 242 & 18 & 253 & 11 & 15 \\
 243 & 18 & 253 & 10 & 15 \\
 244 & 18 & 253 & 9 & 15 \\
 245 & 18 & 253 & 8 & 15 \\
 246 & 18 & 253 & 7 & 15 \\
 247 & 18 & 253 & 6 & 15 \\
 248 & 18 & 253 & 5 & 15 \\
 249 & 18 & 253 & 4 & 15 \\
 250 & 18 & 253 & 3 & 15 \\
 251 & 18 & 253 & 2 & 15 \\
 252 & 18 & 253 & 1 & 15 \\
 253 & 18 & 253 & 0 & 15 \\
 254 & 18 & 253 & 1 & 15 \\
 255 & 18 & 253 & 2 & 15 \\
 256 & 18 & 253 & 3 & 15 \\
 257 & 18 & 253 & 4 & 15 \\
 258 & 18 & 253 & 5 & 15 \\
 259 & 18 & 253 & 6 & 15 \\
 260 & 18 & 253 & 7 & 15 \\
 261 & 18 & 253 & 8 & 15 \\
 262 & 18 & 253 & 9 & 15 \\
 263 & 18 & 253 & 10 & 15 \\
 264 & 18 & 253 & 11 & 15 \\
\end{tabular}
}\end{minipage}
\caption{Values of $z_1$, $|n-b(z_1)|$, and $z_1-3$ for $191 \le n \le 264$}
\label{b-z1 table1}
\end{table}

\begin{table}[H]
\begin{minipage}{0.5\linewidth}\centering
{\footnotesize
\begin{tabular}{ccccc}
$n$ & $z_1$ & $b(z_1)$ & $|n-b(z_1)|$ & $z_1-3$\\
\hline
 265 & 18 & 253 & 12 & 15 \\
 266 & 18 & 253 & 13 & 15 \\
 267 & 18 & 253 & 14 & 15 \\
 268 & 18 & 253 & 15 & 15 \\
 269 & 19 & 285 & 16 & 16 \\
 270 & 19 & 285 & 15 & 16 \\
 271 & 19 & 285 & 14 & 16 \\
 272 & 19 & 285 & 13 & 16 \\
 273 & 19 & 285 & 12 & 16 \\
 274 & 19 & 285 & 11 & 16 \\
 275 & 19 & 285 & 10 & 16 \\
 276 & 19 & 285 & 9 & 16 \\
 277 & 19 & 285 & 8 & 16 \\
 278 & 19 & 285 & 7 & 16 \\
 279 & 19 & 285 & 6 & 16 \\
 280 & 19 & 285 & 5 & 16 \\
 281 & 19 & 285 & 4 & 16 \\
 282 & 19 & 285 & 3 & 16 \\
 283 & 19 & 285 & 2 & 16 \\
 284 & 19 & 285 & 1 & 16 \\
 285 & 19 & 285 & 0 & 16 \\
 286 & 19 & 285 & 1 & 16 \\
 287 & 19 & 285 & 2 & 16 \\
 288 & 19 & 285 & 3 & 16 \\
 289 & 19 & 285 & 4 & 16 \\
 290 & 19 & 285 & 5 & 16 \\
 291 & 19 & 285 & 6 & 16 \\
 292 & 19 & 285 & 7 & 16 \\
 293 & 19 & 285 & 8 & 16 \\
 294 & 19 & 285 & 9 & 16 \\
 295 & 19 & 285 & 10 & 16 \\
 296 & 19 & 285 & 11 & 16 \\
 297 & 19 & 285 & 12 & 16 \\
 298 & 19 & 285 & 13 & 16 \\
 299 & 19 & 285 & 14 & 16 \\
 300 & 19 & 285 & 15 & 16 \\
 301 & 19 & 285 & 16 & 16 \\
 302 & 20 & 319 & 17 & 17 \\
 303 & 20 & 319 & 16 & 17 \\
 304 & 20 & 319 & 15 & 17 \\
 305 & 20 & 319 & 14 & 17 \\
 306 & 20 & 319 & 13 & 17 \\
 307 & 20 & 319 & 12 & 17 \\
 308 & 20 & 319 & 11 & 17 \\
 309 & 20 & 319 & 10 & 17 \\
 310 & 20 & 319 & 9 & 17 \\
 311 & 20 & 319 & 8 & 17 \\
 312 & 20 & 319 & 7 & 17 \\
 313 & 20 & 319 & 6 & 17 \\
 314 & 20 & 319 & 5 & 17 \\
 315 & 20 & 319 & 4 & 17 \\
 316 & 20 & 319 & 3 & 17 \\
 317 & 20 & 319 & 2 & 17 \\
 318 & 20 & 319 & 1 & 17 \\
 319 & 20 & 319 & 0 & 17 \\
\end{tabular}
}\end{minipage}%
\begin{minipage}{0.5\linewidth}\centering
{\footnotesize
\begin{tabular}{ccccc}
$n$ & $z_1$ & $b(z_1)$ & $|n-b(z_1)|$ & $z_1-3$\\
\hline
 320 & 20 & 319 & 1 & 17 \\
 321 & 20 & 319 & 2 & 17 \\
 322 & 20 & 319 & 3 & 17 \\
 323 & 20 & 319 & 4 & 17 \\
 324 & 20 & 319 & 5 & 17 \\
 325 & 20 & 319 & 6 & 17 \\
 326 & 20 & 319 & 7 & 17 \\
 327 & 20 & 319 & 8 & 17 \\
 328 & 20 & 319 & 9 & 17 \\
 329 & 20 & 319 & 10 & 17 \\
 330 & 20 & 319 & 11 & 17 \\
 331 & 20 & 319 & 12 & 17 \\
 332 & 20 & 319 & 13 & 17 \\
 333 & 20 & 319 & 14 & 17 \\
 334 & 20 & 319 & 15 & 17 \\
 335 & 20 & 319 & 16 & 17 \\
 336 & 20 & 319 & 17 & 17 \\
 337 & 21 & 355 & 18 & 18 \\
 338 & 21 & 355 & 17 & 18 \\
 339 & 21 & 355 & 16 & 18 \\
 340 & 21 & 355 & 15 & 18 \\
 341 & 21 & 355 & 14 & 18 \\
 342 & 21 & 355 & 13 & 18 \\
 343 & 21 & 355 & 12 & 18 \\
 344 & 21 & 355 & 11 & 18 \\
 345 & 21 & 355 & 10 & 18 \\
 346 & 21 & 355 & 9 & 18 \\
 347 & 21 & 355 & 8 & 18 \\
 348 & 21 & 355 & 7 & 18 \\
 349 & 21 & 355 & 6 & 18 \\
 350 & 21 & 355 & 5 & 18 \\
 351 & 21 & 355 & 4 & 18 \\
 352 & 21 & 355 & 3 & 18 \\
 353 & 21 & 355 & 2 & 18 \\
 354 & 21 & 355 & 1 & 18 \\
 355 & 21 & 355 & 0 & 18 \\
 356 & 21 & 355 & 1 & 18 \\
 357 & 21 & 355 & 2 & 18 \\
 358 & 21 & 355 & 3 & 18 \\
 359 & 21 & 355 & 4 & 18 \\
 360 & 21 & 355 & 5 & 18 \\
 361 & 21 & 355 & 6 & 18 \\
 362 & 21 & 355 & 7 & 18 \\
 363 & 21 & 355 & 8 & 18 \\
 364 & 21 & 355 & 9 & 18 \\
 365 & 21 & 355 & 10 & 18 \\
 366 & 21 & 355 & 11 & 18 \\
 367 & 21 & 355 & 12 & 18 \\
 368 & 21 & 355 & 13 & 18 \\
 369 & 21 & 355 & 14 & 18 \\
 370 & 21 & 355 & 15 & 18 \\
 371 & 21 & 355 & 16 & 18 \\
 372 & 21 & 355 & 17 & 18 \\
 373 & 21 & 355 & 18 & 18 \\
 374 & 22 & 393 & 19 & 19 \\
\end{tabular}
}\end{minipage}
\caption{Values of $z_1$, $|n-b(z_1)|$, and $z_1-3$ for $265 \le n \le 374$}
\label{b-z1 table2}
\end{table}

\newpage

\begin{table}[H]
\begin{minipage}{0.5\linewidth}\centering
{\footnotesize
\begin{tabular}{ccccc}
$n$ & $z_1$ & $b(z_1)$ & $|n-b(z_1)|$ & $z_1-3$\\
\hline
 375 & 22 & 393 & 18 & 19 \\
 376 & 22 & 393 & 17 & 19 \\
 377 & 22 & 393 & 16 & 19 \\
 378 & 22 & 393 & 15 & 19 \\
 379 & 22 & 393 & 14 & 19 \\
 380 & 22 & 393 & 13 & 19 \\
 381 & 22 & 393 & 12 & 19 \\
 382 & 22 & 393 & 11 & 19 \\
 383 & 22 & 393 & 10 & 19 \\
 384 & 22 & 393 & 9 & 19 \\
 385 & 22 & 393 & 8 & 19 \\
 386 & 22 & 393 & 7 & 19 \\
 387 & 22 & 393 & 6 & 19 \\
 388 & 22 & 393 & 5 & 19 \\
 389 & 22 & 393 & 4 & 19 \\
 390 & 22 & 393 & 3 & 19 \\
 391 & 22 & 393 & 2 & 19 \\
 392 & 22 & 393 & 1 & 19 \\
 393 & 22 & 393 & 0 & 19 \\
 394 & 22 & 393 & 1 & 19 \\
 395 & 22 & 393 & 2 & 19 \\
 396 & 22 & 393 & 3 & 19 \\
 397 & 22 & 393 & 4 & 19 \\
 398 & 22 & 393 & 5 & 19 \\
 399 & 22 & 393 & 6 & 19 \\
 400 & 22 & 393 & 7 & 19 \\
 401 & 22 & 393 & 8 & 19 \\
 402 & 22 & 393 & 9 & 19 \\
 403 & 22 & 393 & 10 & 19 \\
 404 & 22 & 393 & 11 & 19 \\
 405 & 22 & 393 & 12 & 19 \\
 406 & 22 & 393 & 13 & 19 \\
 407 & 22 & 393 & 14 & 19 \\
 408 & 22 & 393 & 15 & 19 \\
 409 & 22 & 393 & 16 & 19 \\
 410 & 22 & 393 & 17 & 19 \\
 411 & 22 & 393 & 18 & 19 \\
 412 & 22 & 393 & 19 & 19 \\
 413 & 23 & 433 & 20 & 20 \\
 414 & 23 & 433 & 19 & 20 \\
 415 & 23 & 433 & 18 & 20 \\
 416 & 23 & 433 & 17 & 20 \\
 417 & 23 & 433 & 16 & 20 \\
 418 & 23 & 433 & 15 & 20 \\
 419 & 23 & 433 & 14 & 20 \\
 420 & 23 & 433 & 13 & 20 \\
 421 & 23 & 433 & 12 & 20 \\
 422 & 23 & 433 & 11 & 20 \\
 423 & 23 & 433 & 10 & 20 \\
 424 & 23 & 433 & 9 & 20 \\
 425 & 23 & 433 & 8 & 20 \\
 426 & 23 & 433 & 7 & 20 \\
 427 & 23 & 433 & 6 & 20 \\
\end{tabular}
}\end{minipage}%
\begin{minipage}{0.5\linewidth}\centering
{\footnotesize
\begin{tabular}{ccccc}
$n$ & $z_1$ & $b(z_1)$ & $|n-b(z_1)|$ & $z_1-3$\\
\hline
 428 & 23 & 433 & 5 & 20 \\
 429 & 23 & 433 & 4 & 20 \\
 430 & 23 & 433 & 3 & 20 \\
 431 & 23 & 433 & 2 & 20 \\
 432 & 23 & 433 & 1 & 20 \\
 433 & 23 & 433 & 0 & 20 \\
 434 & 23 & 433 & 1 & 20 \\
 435 & 23 & 433 & 2 & 20 \\
 436 & 23 & 433 & 3 & 20 \\
 437 & 23 & 433 & 4 & 20 \\
 438 & 23 & 433 & 5 & 20 \\
 439 & 23 & 433 & 6 & 20 \\
 440 & 23 & 433 & 7 & 20 \\
 441 & 23 & 433 & 8 & 20 \\
 442 & 23 & 433 & 9 & 20 \\
 443 & 23 & 433 & 10 & 20 \\
 444 & 23 & 433 & 11 & 20 \\
 445 & 23 & 433 & 12 & 20 \\
 446 & 23 & 433 & 13 & 20 \\
 447 & 23 & 433 & 14 & 20 \\
 448 & 23 & 433 & 15 & 20 \\
 449 & 23 & 433 & 16 & 20 \\
 450 & 23 & 433 & 17 & 20 \\
 451 & 23 & 433 & 18 & 20 \\
 452 & 23 & 433 & 19 & 20 \\
 453 & 23 & 433 & 20 & 20 \\
 454 & 24 & 475 & 21 & 21 \\
 455 & 24 & 475 & 20 & 21 \\
 456 & 24 & 475 & 19 & 21 \\
 457 & 24 & 475 & 18 & 21 \\
 458 & 24 & 475 & 17 & 21 \\
 459 & 24 & 475 & 16 & 21 \\
 460 & 24 & 475 & 15 & 21 \\
 461 & 24 & 475 & 14 & 21 \\
 462 & 24 & 475 & 13 & 21 \\
 463 & 24 & 475 & 12 & 21 \\
 464 & 24 & 475 & 11 & 21 \\
 465 & 24 & 475 & 10 & 21 \\
 466 & 24 & 475 & 9 & 21 \\
 467 & 24 & 475 & 8 & 21 \\
 468 & 24 & 475 & 7 & 21 \\
 469 & 24 & 475 & 6 & 21 \\
 470 & 24 & 475 & 5 & 21 \\
 471 & 24 & 475 & 4 & 21 \\
 472 & 24 & 475 & 3 & 21 \\
 473 & 24 & 475 & 2 & 21 \\
 474 & 24 & 475 & 1 & 21 \\
 475 & 24 & 475 & 0 & 21 \\
 476 & 24 & 475 & 1 & 21 \\
 477 & 24 & 475 & 2 & 21 \\
 478 & 24 & 475 & 3 & 21 \\
 479 & 24 & 475 & 4 & 21 \\
 480 & 24 & 475 & 5 & 21 \\
\end{tabular}
}\end{minipage}%
\caption{Values of $z_1$, $|n-b(z_1)|$, and $z_1-3$ for $375 \le n \le 480$}
\label{b-z1 table3}
\end{table}

\subsection{Code for Theorem \ref{thm:restimate}}
\label{app:4.6}
For all $n \ge 3$, define $N:=\lfloor \log_2 \log_5 n \rfloor+1$ and 
\begin{equation*}
r(n) = \left( \sum_{j=1}^N \dfrac{1}{2^{j-1}} n^{\frac{1}{2^j}}\right) - \dfrac{2^{N-1}-1}{2^{N-1}}
\end{equation*}
Theorem \ref{thm:restimate} asserts that $|\zeta(n) - r(n)| < 1.985$ for all $n \ge 3$. The first portion of the proof requires verifying this directly for all $3 \le n < 3350194786$. Table \ref{zeta and r table} summarizes the relevant data for the exceptional cases $n=5^{2^d}$ and $n=5^{2^d}-1$ for $1 \le d \le 3$.

\begin{table}[H]
\begin{minipage}{0.5\linewidth}\centering\centering
{\footnotesize
\begin{tabular}{cccc}
$n$ & $\zeta(n)$ & $r(n)$ & $|\zeta(n)-r(n)|$\\
\hline
 $5^2-1$ & 5 & 4.89898 & 0.101021 \\
 $5^2$ & 6 & 5.61803 & 0.381966 \\
 $5^4-1$ & 28 & 26.979 & 1.02101 \\
\end{tabular}
}\end{minipage}%
\begin{minipage}{0.5\linewidth}\centering 
 {\footnotesize
\begin{tabular}{cccc}
$n$ & $\zeta(n)$ & $r(n)$ & $|\zeta(n)-r(n)|$\\
\hline
 $5^4$ & 28 & 27.309 & 0.690983 \\
 $5^8-1$ & 639 & 637.999 & 1.00081 \\
 $5^8$ & 639 & 638.155 & 0.845492
\end{tabular}
}\end{minipage}
\caption{Values of $\zeta(n)$, $r(n)$, and errors for $n=5^{2^d}-1$, $5^{2^d}$ with $1 \le d \le 3$}
\label{zeta and r table}
\end{table}

For the remaining values of $n$, recall that if $n \in [\ell(m), \ell(m+1)-1]$, then $\zeta(n)=m$. Now, one may compute that $\ell(58006) = 3350194786$. This means that, rather than check every value of $n$ less than 3350194786, it suffices to work with $n=\ell(k)$ and $n=\ell(k+1)-1$ for $1 \le k \le 58005$.

The Mathematica code below computes the required values of $\ell(k)$ and $\ell(k+1)-1$, and also evaluates $r(\ell(k))$ and $r(\ell(k+1)-1)$. The values of $\ell$ come from \texttt{LTable}, which can be produced by the code given for computing $b(n)$ in Appendix \ref{app:bcode}.
 
\begin{verbatim}
BigN[n_] := Floor[Log[2, Log[5, n]]] + 1;
r[n_] := Sum[1/(2^(j - 1))*n^(1/(2^j)), {j, 1, BigN[n]}] 
         - (2^(BigN[n] - 1) - 1)/2^(BigN[n] - 1);
nTestVals = 
 Flatten[Table[{LTable[[k]], LTable[[k + 1]] - 1}, {k, 1, 58005}]];
rVals = Table[N[r[nTestVals[[k]]]], {k, 1, 2*58005}];
\end{verbatim}

With the above code, it is not necessary to compute $\zeta(n)$ again in order to evaluate $|\zeta(n)-r(n)|$. This is because the test values for $n$ are stored in \texttt{nTestVals} in the following order:
\begin{equation*}
\ell(1), \ell(2)-1, \ell(2), \ell(3)-1, \ldots, \ell(m), \ell(m+1)-1, \ldots, \ell(58005), \ell(58006)-1
\end{equation*}
Applying $\zeta$ to each element of this list produces, in order:
\begin{equation*}
1,1,2,2,\ldots,m,m,\ldots,58005,58005.
\end{equation*}
In this list, the number at index $i$ is equal to $\lceil \frac{i}{2} \rceil$. 
Thus, the error between $r(n)$ and $\zeta(n)$ can be found by comparing \texttt{rVals[[i]]} and \texttt{Ceiling[i/2]}. The command below calculates the maximum error that occurs in this range.
\begin{verbatim}
maxError = 
   Max[Table[Abs[rVals[[i]] - Ceiling[i/2]], {i, 1, 2*58005}]]
\end{verbatim}
This returns a (rounded) value of \texttt{1.45175}.

\subsection{Data for Lemma \ref{lem:betterzrootndiff}}
\label{app:5.1}
Note that if $n \in [\ell(m), \ell(m+1)-1]$, then $\zeta(n)=m$. So, to establish upper bounds on $z_1 := \zeta(n)$, it suffices to compare $z_1$ to $u(z_1) := \lfloor \sqrt{\ell(z_1)} \rfloor + \ell(z_1)^{\frac{1}{4}}$.

\begin{table}[H]
\begin{minipage}{0.45\linewidth}\centering
{\footnotesize
\begin{tabular}{cccc}
$n$ & $z_1$ & $\ell(z_1)$ & $u(z_1)$\\
\hline
 $[3,4]$ & 1 & 3 & 2.316 \\
 $[5,6]$ & 2 & 5 & 3.495 \\
 $[7,11]$ & 3 & 7 & 3.627 \\
 $[12,16]$ & 4 & 12 & 4.861 \\
 $[17,24]$ & 5 & 17 & 6.031 \\
 $[25,30]$ & 6 & 25 & 7.236 \\
 $[31,46]$ & 7 & 31 & 7.360 \\
 $[47,54]$ & 8 & 47 & 8.618 \\
 $[55,69]$ & 9 & 55 & 9.723 \\
 $[70,86]$ & 10 & 70 & 10.89 \\
 $[87,99]$ & 11 & 87 & 12.05 \\
 $[100,119]$ & 12 & 100 & 13.16 \\
 $[120,141]$ & 13 & 120 & 13.31 \\
 $[142,165]$ & 14 & 142 & 14.45 \\
 $[166,191]$ & 15 & 166 & 15.59 \\
 $[192,208]$ & 16 & 192 & 16.72 \\
 $[209,237]$ & 17 & 209 & 17.80 \\
 $[238,268]$ & 18 & 238 & 18.93 \\
 $[269,301]$ & 19 & 269 & 20.05 \\ 
 $[302,336]$ & 20 & 302 & 21.17 \\
 $[337,373]$ & 21 & 337 & 22.28 \\
 $[374,412]$ & 22 & 374 & 23.40 \\
 $[413,453]$ & 23 & 413 & 24.51 \\
 $[454,480]$ & 24 & 454 & 25.62 \\
 $[481,524]$ & 25 & 481 & 25.68 \\
 $[525,570]$ & 26 & 525 & 26.79 \\
 $[571,618]$ & 27 & 571 & 27.89 \\
 $[619,668]$ & 28 & 619 & 28.99 \\
 $[669,720]$ & 29 & 669 & 30.09 \\
 $[721,750]$ & 30 & 721 & 31.18 \\
 $[751,805]$ & 31 & 751 & 32.23 \\
 $[806,862]$ & 32 & 806 & 33.33 \\
 $[863,921]$ & 33 & 863 & 34.42 \\
 $[922,982]$ & 34 & 922 & 35.51 \\
 $[983,1045]$ & 35 & 983 & 36.60 \\
 $[1046,1110]$ & 36 & 1046 & 37.69 \\
 $[1111,1177]$ & 37 & 1111 & 38.77 \\
 $[1178,1246]$ & 38 & 1178 & 39.86 \\
  $[1247,1317]$ & 39 & 1247 & 40.94 \\
 $[1318,1390]$ & 40 & 1318 & 42.03 \\
$[1391,1465]$ & 41 & 1391 & 43.11 \\
 $[1466,1542]$ & 42 & 1466 & 44.19 \\
 $[1543,1621]$ & 43 & 1543 & 45.27 \\
\end{tabular}
}\end{minipage}%
\hspace{0.2in}
\begin{minipage}{0.45\linewidth}\centering
{\footnotesize
\begin{tabular}{ccccc}
$n$ & $z_1$ & $\ell(z_1)$ & $u(z_1)$\\
\hline
 $[1622,1702]$ & 44 & 1622 & 46.35 \\
 $[1703,1785]$ & 45 & 1703 & 47.42 \\
 $[1786,1840]$ & 46 & 1786 & 48.50 \\
 $[1841,1926]$ & 47 & 1841 & 48.55 \\
 $[1927,2014]$ & 48 & 1927 & 49.63 \\
 $[2015,2104]$ & 49 & 2015 & 50.70 \\
 $[2105,2196]$ & 50 & 2105 & 51.77 \\
 $[2197,2290]$ & 51 & 2197 & 52.85 \\
 $[2291,2386]$ & 52 & 2291 & 53.92 \\
 $[2387,2484]$ & 53 & 2387 & 54.99 \\
 $[2485,2538]$ & 54 & 2485 & 56.06 \\
 $[2539,2639]$ & 55 & 2539 & 57.10 \\
 $[2640,2742]$ & 56 & 2640 & 58.17 \\
 $[2743,2847]$ & 57 & 2743 & 59.24 \\
 $[2848,2954]$ & 58 & 2848 & 60.31 \\
 $[2955,3063]$ & 59 & 2955 & 61.37 \\
 $[3064,3174]$ & 60 & 3064 & 62.44 \\
 $[3175,3287]$ & 61 & 3175 & 63.51 \\
 $[3288,3402]$ & 62 & 3288 & 64.57 \\
 $[3403,3519]$ & 63 & 3403 & 65.64 \\
 $[3520,3638]$ & 64 & 3520 & 66.70 \\
 $[3639,3759]$ & 65 & 3639 & 67.77 \\
 $[3760,3882]$ & 66 & 3760 & 68.83 \\
 $[3883,4007]$ & 67 & 3883 & 69.89 \\
 $[4008,4134]$ & 68 & 4008 & 70.96 \\
 $[4135,4209]$ & 69 & 4135 & 72.02 \\
 $[4210,4339]$ & 70 & 4210 & 72.06 \\
 $[4340,4471]$ & 71 & 4340 & 73.12 \\
 $[4472,4605]$ & 72 & 4472 & 74.18 \\
 $[4606,4741]$ & 73 & 4606 & 75.24 \\
 $[4742,4879]$ & 74 & 4742 & 76.30 \\
 $[4880,5019]$ & 75 & 4880 & 77.36 \\
 $[5020,5161]$ & 76 & 5020 & 78.42 \\
 $[5162,5305]$ & 77 & 5162 & 79.48 \\
 $[5306,5451]$ & 78 & 5306 & 80.53 \\
 $[5452,5599]$ & 79 & 5452 & 81.59 \\
 $[5600,5749]$ & 80 & 5600 & 82.65 \\
  $[5750,5901]$ & 81 & 5750 & 83.71 \\
 $[5902,6055]$ & 82 & 5902 & 84.76 \\
 $[6056,6211]$ & 83 & 6056 & 85.82 \\
 $[6212,6369]$ & 84 & 6212 & 86.88 \\
 $[6370,6529]$ & 85 & 6370 & 87.93 \\
 $[6530,6622]$ & 86 & 6530 & 88.99 \\
\end{tabular}
}\end{minipage}
\caption{Values of $z_1$, $\ell(z_1)$, and $u(z_1) := \lfloor \sqrt{\ell(z_1)} \rfloor + \ell(z_1)^{\frac{1}{4}}$ for $3 \le n \le 6622$}
\label{z1 sqrt table}
\end{table}

\subsection{Data for Lemma \ref{lem:convergence} and Theorem \ref{thm:Westimate}}
\label{app:5.3} 
Define the following functions.
\begin{itemize}
\item For each $n \ge 3$, $h(n) = b(n) - (n-1) + \displaystyle\sum_{k=1}^{z_1} |W(k)|$
\item For each real number $x \ge 0$, $\lastFunc(x) = x^2 - \displaystyle\sum_{k=1}^\infty \left(\dfrac{2^k}{\prod_{i=1}^k (2^i+1)}\right)x^{1+\tfrac{1}{2^k}}$
\end{itemize}

In Tables \ref{Size W(n) table}, \ref{Size W(n) table1}, \ref{Size W(n) table2}, \ref{Size W(n) table3}, and \ref{Size W(n) table4}, for each value of $n$, $e_1$ represents $||W(n)|-h(n)|$ and $e_2$ represents $||W(n)|-\lastFunc(n)|$.

\begin{table}[H]
\begin{minipage}{0.48\linewidth}\centering
{\scriptsize
\begin{tabular}{cccccc}
$n$ & $|W(n)|$ & $h(n)$ & $e_1$ & $\lastFunc(n)$ & $e_2$\\
\hline
 3 & 6 & 6 & 0 & 4.26 & 1.74 \\
 4 & 11 & 11 & 0 & 8.84 & 2.16 \\
 5 & 17 & 19 & 2 & 15.15 & 1.85 \\
 6 & 25 & 27 & 2 & 23.20 & 1.80 \\
 7 & 34 & 39 & 5 & 33.03 & 0.97 \\
 8 & 47 & 55 & 8 & 44.65 & 2.35 \\
 9 & 59 & 63 & 4 & 58.06 & 0.94 \\
 10 & 74 & 78 & 4 & 73.30 & 0.70 \\
 11 & 91 & 95 & 4 & 90.36 & 0.64 \\
 12 & 109 & 119 & 10 & 109.26 & 0.26 \\
 13 & 129 & 139 & 10 & 130.00 & 1.0 \\
 14 & 152 & 161 & 9 & 152.58 & 0.58 \\
 15 & 176 & 185 & 9 & 177.02 & 1.02 \\
 16 & 202 & 211 & 9 & 203.32 & 1.32 \\
 17 & 229 & 245 & 16 & 231.48 & 2.48 \\
 18 & 259 & 274 & 15 & 261.51 & 2.51 \\
 19 & 290 & 305 & 15 & 293.41 & 3.41 \\
 20 & 323 & 338 & 15 & 327.19 & 4.19 \\
 21 & 358 & 373 & 15 & 362.85 & 4.85 \\
 22 & 396 & 410 & 14 & 400.39 & 4.39 \\
 23 & 435 & 449 & 14 & 439.81 & 4.81 \\
 24 & 476 & 490 & 14 & 481.12 & 5.12 \\
 25 & 518 & 542 & 24 & 524.32 & 6.32 \\
 26 & 562 & 586 & 24 & 569.41 & 7.41 \\
 27 & 609 & 632 & 23 & 616.40 & 7.40 \\
 28 & 657 & 680 & 23 & 665.28 & 8.28 \\
 29 & 707 & 730 & 23 & 716.06 & 9.06 \\
 30 & 759 & 782 & 23 & 768.74 & 9.74 \\
 31 & 812 & 846 & 34 & 823.32 & 11.32 \\
 32 & 869 & 901 & 32 & 879.81 & 10.81 \\
 33 & 927 & 958 & 31 & 938.20 & 11.20 \\
 34 & 987 & 1017 & 30 & 998.50 & 11.50 \\
 35 & 1048 & 1078 & 30 & 1060.70 & 12.70 \\
 36 & 1111 & 1141 & 30 & 1124.81 & 13.81 \\
 37 & 1176 & 1206 & 30 & 1190.84 & 14.84 \\
 38 & 1244 & 1273 & 29 & 1258.78 & 14.78 \\
 39 & 1313 & 1342 & 29 & 1328.63 & 15.63 \\
 40 & 1384 & 1413 & 29 & 1400.39 & 16.39 \\
 41 & 1457 & 1486 & 29 & 1474.07 & 17.07 \\
 42 & 1532 & 1561 & 29 & 1549.67 & 17.67 \\
 43 & 1609 & 1638 & 29 & 1627.18 & 18.18 \\
 44 & 1689 & 1717 & 28 & 1706.61 & 17.61 \\
 45 & 1770 & 1798 & 28 & 1787.96 & 17.96 \\
 46 & 1853 & 1881 & 28 & 1871.23 & 18.23 \\
 47 & 1937 & 1983 & 46 & 1956.43 & 19.43 \\
 48 & 2023 & 2069 & 46 & 2043.54 & 20.54 \\
 49 & 2111 & 2157 & 46 & 2132.58 & 21.58 \\
 50 & 2201 & 2247 & 46 & 2223.54 & 22.54 \\
 51 & 2294 & 2339 & 45 & 2316.43 & 22.43 \\
 52 & 2388 & 2433 & 45 & 2411.24 & 23.24 \\
 53 & 2484 & 2529 & 45 & 2507.98 & 23.98 \\
 54 & 2582 & 2627 & 45 & 2606.64 & 24.64 \\
\end{tabular}
}\end{minipage}
\begin{minipage}{0.48\textwidth}\centering
{\scriptsize
\begin{tabular}{cccccc}
$n$ & $|W(n)|$ & $h(n)$ & $e_1$ & $\lastFunc(n)$ & $e_2$\\
\hline
 55 & 2681 & 2740 & 59 & 2707.23 & 26.23 \\
 56 & 2783 & 2841 & 58 & 2809.75 & 26.75 \\
 57 & 2886 & 2944 & 58 & 2914.20 & 28.20 \\
 58 & 2993 & 3049 & 56 & 3020.58 & 27.58 \\
 59 & 3101 & 3156 & 55 & 3128.89 & 27.89 \\
 60 & 3211 & 3265 & 54 & 3239.13 & 28.13 \\
 61 & 3322 & 3376 & 54 & 3351.30 & 29.30 \\
 62 & 3435 & 3489 & 54 & 3465.40 & 30.40 \\
 63 & 3550 & 3604 & 54 & 3581.44 & 31.44 \\
 64 & 3667 & 3721 & 54 & 3699.41 & 32.41 \\
 65 & 3786 & 3840 & 54 & 3819.31 & 33.31 \\
 66 & 3908 & 3961 & 53 & 3941.14 & 33.14 \\
 67 & 4031 & 4084 & 53 & 4064.91 & 33.91 \\
 68 & 4156 & 4209 & 53 & 4190.62 & 34.62 \\
 69 & 4283 & 4336 & 53 & 4318.26 & 35.26 \\
 70 & 4411 & 4485 & 74 & 4447.84 & 36.84 \\
 71 & 4542 & 4615 & 73 & 4579.35 & 37.35 \\
 72 & 4674 & 4747 & 73 & 4712.80 & 38.80 \\
 73 & 4808 & 4881 & 73 & 4848.19 & 40.19 \\
 74 & 4946 & 5017 & 71 & 4985.52 & 39.52 \\
 75 & 5085 & 5155 & 70 & 5124.78 & 39.78 \\
 76 & 5226 & 5295 & 69 & 5265.99 & 39.99 \\
 77 & 5368 & 5437 & 69 & 5409.13 & 41.13 \\
 78 & 5512 & 5581 & 69 & 5554.21 & 42.21 \\
 79 & 5658 & 5727 & 69 & 5701.24 & 43.24 \\
 80 & 5806 & 5875 & 69 & 5850.20 & 44.20 \\
 81 & 5956 & 6025 & 69 & 6001.10 & 45.10 \\
 82 & 6108 & 6177 & 69 & 6153.95 & 45.95 \\
 83 & 6263 & 6331 & 68 & 6308.73 & 45.73 \\
 84 & 6419 & 6487 & 68 & 6465.46 & 46.46 \\
 85 & 6577 & 6645 & 68 & 6624.13 & 47.13 \\
 86 & 6737 & 6805 & 68 & 6784.75 & 47.75 \\
 87 & 6898 & 6989 & 91 & 6947.30 & 49.30 \\
 88 & 7062 & 7152 & 90 & 7111.80 & 49.80 \\
 89 & 7227 & 7317 & 90 & 7278.24 & 51.24 \\
 90 & 7394 & 7484 & 90 & 7446.63 & 52.63 \\
 91 & 7563 & 7653 & 90 & 7616.96 & 53.96 \\
 92 & 7736 & 7824 & 88 & 7789.24 & 53.24 \\
 93 & 7910 & 7997 & 87 & 7963.46 & 53.46 \\
 94 & 8086 & 8172 & 86 & 8139.63 & 53.63 \\
 95 & 8263 & 8349 & 86 & 8317.74 & 54.74 \\
 96 & 8442 & 8528 & 86 & 8497.79 & 55.79 \\
 97 & 8623 & 8709 & 86 & 8679.80 & 56.80 \\
 98 & 8806 & 8892 & 86 & 8863.75 & 57.75 \\
 99 & 8991 & 9077 & 86 & 9049.64 & 58.64 \\
 100 & 9177 & 9287 & 110 & 9237.49 & 60.49 \\
 101 & 9366 & 9475 & 109 & 9427.28 & 61.28 \\
 102 & 9558 & 9665 & 107 & 9619.02 & 61.02 \\
 103 & 9751 & 9857 & 106 & 9812.70 & 61.70 \\
 104 & 9946 & 10051 & 105 & 10008.34 & 62.34 \\
 105 & 10143 & 10247 & 104 & 10205.92 & 62.92 \\
 106 & 10341 & 10445 & 104 & 10405.45 & 64.45 \\
\end{tabular}
}\end{minipage}
\caption{Values of $|W(n)|$, $h(n)$, $\lastFunc(n)$, and errors for $3 \le n \le 106$}
\label{Size W(n) table}
\end{table}

\newpage

\begin{table}[H]
\begin{minipage}{0.48\linewidth}\centering
{\scriptsize
\begin{tabular}{cccccc}
$n$ & $|W(n)|$ & $h(n)$ & $e_1$ & $\lastFunc(n)$ & $e_2$\\
\hline
 107 & 10542 & 10645 & 103 & 10606.93 & 64.93 \\
 108 & 10744 & 10847 & 103 & 10810.36 & 66.36 \\
 109 & 10948 & 11051 & 103 & 11015.73 & 67.73 \\
 110 & 11154 & 11257 & 103 & 11223.06 & 69.06 \\
 111 & 11362 & 11465 & 103 & 11432.34 & 70.34 \\
 112 & 11574 & 11675 & 101 & 11643.57 & 69.57 \\
 113 & 11787 & 11887 & 100 & 11856.74 & 69.74 \\
 114 & 12002 & 12101 & 99 & 12071.87 & 69.87 \\
 115 & 12218 & 12317 & 99 & 12288.95 & 70.95 \\
 116 & 12436 & 12535 & 99 & 12507.98 & 71.98 \\
 117 & 12656 & 12755 & 99 & 12728.96 & 72.96 \\
 118 & 12878 & 12977 & 99 & 12951.89 & 73.89 \\
 119 & 13102 & 13201 & 99 & 13176.77 & 74.77 \\
 120 & 13327 & 13457 & 130 & 13403.61 & 76.61 \\
 121 & 13554 & 13684 & 130 & 13632.39 & 78.39 \\
 122 & 13784 & 13913 & 129 & 13863.13 & 79.13 \\
 123 & 14017 & 14144 & 127 & 14095.82 & 78.82 \\
 124 & 14251 & 14377 & 126 & 14330.47 & 79.47 \\
 125 & 14487 & 14612 & 125 & 14567.06 & 80.06 \\
 126 & 14725 & 14849 & 124 & 14805.61 & 80.61 \\
 127 & 14964 & 15088 & 124 & 15046.11 & 82.11 \\
 128 & 15206 & 15329 & 123 & 15288.57 & 82.57 \\
 129 & 15449 & 15572 & 123 & 15532.98 & 83.98 \\
 130 & 15694 & 15817 & 123 & 15779.34 & 85.34 \\
 131 & 15941 & 16064 & 123 & 16027.66 & 86.66 \\
 132 & 16190 & 16313 & 123 & 16277.93 & 87.93 \\
 133 & 16441 & 16564 & 123 & 16530.15 & 89.15 \\
 134 & 16696 & 16817 & 121 & 16784.33 & 88.33 \\
 135 & 16952 & 17072 & 120 & 17040.47 & 88.47 \\
 136 & 17210 & 17329 & 119 & 17298.55 & 88.55 \\
 137 & 17469 & 17588 & 119 & 17558.60 & 89.60 \\
 138 & 17730 & 17849 & 119 & 17820.60 & 90.60 \\
 139 & 17993 & 18112 & 119 & 18084.55 & 91.55 \\
 140 & 18258 & 18377 & 119 & 18350.46 & 92.46 \\
 141 & 18525 & 18644 & 119 & 18618.32 & 93.32 \\
 142 & 18793 & 18946 & 153 & 18888.14 & 95.14 \\
 143 & 19063 & 19216 & 153 & 19159.92 & 96.92 \\
 144 & 19336 & 19488 & 152 & 19433.65 & 97.65 \\
 145 & 19611 & 19762 & 151 & 19709.34 & 98.34 \\
 146 & 19889 & 20038 & 149 & 19986.99 & 97.99 \\
 147 & 20168 & 20316 & 148 & 20266.59 & 98.59 \\
 148 & 20449 & 20596 & 147 & 20548.15 & 99.15 \\
 149 & 20732 & 20878 & 146 & 20831.66 & 99.66 \\
 150 & 21016 & 21162 & 146 & 21117.13 & 101.13 \\
 151 & 21303 & 21448 & 145 & 21404.56 & 101.56 \\
 152 & 21591 & 21736 & 145 & 21693.95 & 102.95 \\
 153 & 21881 & 22026 & 145 & 21985.29 & 104.29 \\
 154 & 22173 & 22318 & 145 & 22278.60 & 105.60 \\
 155 & 22467 & 22612 & 145 & 22573.85 & 106.85 \\
 156 & 22763 & 22908 & 145 & 22871.07 & 108.07 \\
 157 & 23061 & 23206 & 145 & 23170.25 & 109.25 \\
 158 & 23363 & 23506 & 143 & 23471.38 & 108.38 \\
 159 & 23666 & 23808 & 142 & 23774.47 & 108.47 \\
 160 & 23971 & 24112 & 141 & 24079.52 & 108.52 \\
 161 & 24277 & 24418 & 141 & 24386.53 & 109.53 \\
 162 & 24585 & 24726 & 141 & 24695.50 & 110.50 \\
 163 & 24895 & 25036 & 141 & 25006.42 & 111.42 \\
 164 & 25207 & 25348 & 141 & 25319.31 & 112.31 \\
 165 & 25521 & 25662 & 141 & 25634.15 & 113.15 \\
 166 & 25836 & 26013 & 177 & 25950.95 & 114.95 \\
 167 & 26153 & 26330 & 177 & 26269.71 & 116.71 \\
 168 & 26473 & 26649 & 176 & 26590.44 & 117.44 \\
 169 & 26794 & 26970 & 176 & 26913.12 & 119.12 \\
 170 & 27118 & 27293 & 175 & 27237.76 & 119.76 \\
 171 & 27445 & 27618 & 173 & 27564.36 & 119.36 \\
\end{tabular}
}\end{minipage}
\begin{minipage}{0.48\textwidth}\centering
{\scriptsize
\begin{tabular}{cccccc}
$n$ & $|W(n)|$ & $h(n)$ & $e_1$ & $\lastFunc(n)$ & $e_2$\\
\hline
 172 & 27773 & 27945 & 172 & 27892.92 & 119.92 \\
 173 & 28103 & 28274 & 171 & 28223.44 & 120.44 \\
 174 & 28435 & 28605 & 170 & 28555.92 & 120.92 \\
 175 & 28768 & 28938 & 170 & 28890.36 & 122.36 \\
 176 & 29104 & 29273 & 169 & 29226.76 & 122.76 \\
 177 & 29441 & 29610 & 169 & 29565.12 & 124.12 \\
 178 & 29780 & 29949 & 169 & 29905.45 & 125.45 \\
 179 & 30121 & 30290 & 169 & 30247.73 & 126.73 \\
 180 & 30464 & 30633 & 169 & 30591.97 & 127.97 \\
 181 & 30809 & 30978 & 169 & 30938.18 & 129.18 \\
 182 & 31156 & 31325 & 169 & 31286.34 & 130.34 \\
 183 & 31505 & 31674 & 169 & 31636.47 & 131.47 \\
 184 & 31858 & 32025 & 167 & 31988.56 & 130.56 \\
 185 & 32212 & 32378 & 166 & 32342.61 & 130.61 \\
 186 & 32568 & 32733 & 165 & 32698.62 & 130.62 \\
 187 & 32925 & 33090 & 165 & 33056.59 & 131.59 \\
 188 & 33284 & 33449 & 165 & 33416.53 & 132.53 \\
 189 & 33645 & 33810 & 165 & 33778.42 & 133.42 \\
 190 & 34008 & 34173 & 165 & 34142.28 & 134.28 \\
 191 & 34373 & 34538 & 165 & 34508.10 & 135.10 \\
 192 & 34739 & 34942 & 203 & 34875.88 & 136.88 \\
 193 & 35107 & 35310 & 203 & 35245.63 & 138.63 \\
 194 & 35478 & 35680 & 202 & 35617.33 & 139.33 \\
 195 & 35850 & 36052 & 202 & 35991.00 & 141.00 \\
 196 & 36224 & 36426 & 202 & 36366.63 & 142.63 \\
 197 & 36601 & 36802 & 201 & 36744.23 & 143.23 \\
 198 & 36981 & 37180 & 199 & 37123.79 & 142.79 \\
 199 & 37362 & 37560 & 198 & 37505.31 & 143.31 \\
 200 & 37745 & 37942 & 197 & 37888.79 & 143.79 \\
 201 & 38130 & 38326 & 196 & 38274.23 & 144.23 \\
 202 & 38516 & 38712 & 196 & 38661.64 & 145.64 \\
 203 & 38905 & 39100 & 195 & 39051.02 & 146.02 \\
 204 & 39295 & 39490 & 195 & 39442.35 & 147.35 \\
 205 & 39687 & 39882 & 195 & 39835.65 & 148.65 \\
 206 & 40081 & 40276 & 195 & 40230.91 & 149.91 \\
 207 & 40477 & 40672 & 195 & 40628.14 & 151.14 \\
 208 & 40875 & 41070 & 195 & 41027.33 & 152.33 \\
 209 & 41274 & 41508 & 234 & 41428.48 & 154.48 \\
 210 & 41676 & 41909 & 233 & 41831.60 & 155.60 \\
 211 & 42080 & 42312 & 232 & 42236.68 & 156.68 \\
 212 & 42488 & 42717 & 229 & 42643.72 & 155.72 \\
 213 & 42897 & 43124 & 227 & 43052.73 & 155.73 \\
 214 & 43308 & 43533 & 225 & 43463.70 & 155.70 \\
 215 & 43720 & 43944 & 224 & 43876.64 & 156.64 \\
 216 & 44134 & 44357 & 223 & 44291.54 & 157.54 \\
 217 & 44550 & 44772 & 222 & 44708.41 & 158.41 \\
 218 & 44968 & 45189 & 221 & 45127.24 & 159.24 \\
 219 & 45388 & 45608 & 220 & 45548.03 & 160.03 \\
 220 & 45809 & 46029 & 220 & 45970.79 & 161.79 \\
 221 & 46232 & 46452 & 220 & 46395.52 & 163.52 \\
 222 & 46658 & 46877 & 219 & 46822.21 & 164.21 \\
 223 & 47085 & 47304 & 219 & 47250.86 & 165.86 \\
 224 & 47514 & 47733 & 219 & 47681.48 & 167.48 \\
 225 & 47945 & 48164 & 219 & 48114.06 & 169.06 \\
 226 & 48379 & 48597 & 218 & 48548.61 & 169.61 \\
 227 & 48816 & 49032 & 216 & 48985.13 & 169.13 \\
 228 & 49254 & 49469 & 215 & 49423.61 & 169.61 \\
 229 & 49694 & 49908 & 214 & 49864.05 & 170.05 \\
 230 & 50136 & 50349 & 213 & 50306.47 & 170.47 \\
 231 & 50579 & 50792 & 213 & 50750.84 & 171.84 \\
 232 & 51025 & 51237 & 212 & 51197.18 & 172.18 \\
 233 & 51472 & 51684 & 212 & 51645.49 & 173.49 \\
 234 & 51921 & 52133 & 212 & 52095.77 & 174.77 \\
 235 & 52372 & 52584 & 212 & 52548.00 & 176.00 \\
 236 & 52825 & 53037 & 212 & 53002.21 & 177.21 \\
\end{tabular}
}\end{minipage}
\caption{Values of $|W(n)|$, $h(n)$, $\lastFunc(n)$, and errors for $107 \le n \le 236$}
\label{Size W(n) table1}
\end{table}

\newpage

\begin{table}[H]
\begin{minipage}{0.5\linewidth}\centering
{\scriptsize
\begin{tabular}{cccccc}
$n$ & $|W(n)|$ & $h(n)$ & $e_1$ & $\lastFunc(n)$ & $e_2$\\
\hline
 237 & 53280 & 53492 & 212 & 53458.38 & 178.38 \\
 238 & 53736 & 54000 & 264 & 53916.52 & 180.52 \\
 239 & 54195 & 54458 & 263 & 54376.62 & 181.62 \\
 240 & 54656 & 54918 & 262 & 54838.69 & 182.69 \\
 241 & 55119 & 55380 & 261 & 55302.73 & 183.73 \\
 242 & 55586 & 55844 & 258 & 55768.73 & 182.73 \\
 243 & 56054 & 56310 & 256 & 56236.70 & 182.70 \\
 244 & 56524 & 56778 & 254 & 56706.63 & 182.63 \\
 245 & 56995 & 57248 & 253 & 57178.54 & 183.54 \\
 246 & 57468 & 57720 & 252 & 57652.40 & 184.40 \\
 247 & 57943 & 58194 & 251 & 58128.24 & 185.24 \\
 248 & 58420 & 58670 & 250 & 58606.04 & 186.04 \\
 249 & 58899 & 59148 & 249 & 59085.81 & 186.81 \\
 250 & 59379 & 59628 & 249 & 59567.55 & 188.55 \\
 251 & 59861 & 60110 & 249 & 60051.25 & 190.25 \\
 252 & 60346 & 60594 & 248 & 60536.92 & 190.92 \\
 253 & 60832 & 61080 & 248 & 61024.56 & 192.56 \\
 254 & 61320 & 61568 & 248 & 61514.16 & 194.16 \\
 255 & 61810 & 62058 & 248 & 62005.73 & 195.73 \\
 256 & 62302 & 62550 & 248 & 62499.27 & 197.27 \\
 257 & 62797 & 63044 & 247 & 62994.78 & 197.78 \\
 258 & 63295 & 63540 & 245 & 63492.25 & 197.25 \\
 259 & 63794 & 64038 & 244 & 63991.69 & 197.69 \\
 260 & 64295 & 64538 & 243 & 64493.10 & 198.10 \\
 261 & 64798 & 65040 & 242 & 64996.47 & 198.47 \\
 262 & 65302 & 65544 & 242 & 65501.82 & 199.82 \\
 263 & 65809 & 66050 & 241 & 66009.13 & 200.13 \\
 264 & 66317 & 66558 & 241 & 66518.41 & 201.41 \\
 265 & 66827 & 67068 & 241 & 67029.66 & 202.66 \\
 266 & 67339 & 67580 & 241 & 67542.87 & 203.87 \\
 267 & 67853 & 68094 & 241 & 68058.06 & 205.06 \\
 268 & 68369 & 68610 & 241 & 68575.21 & 206.21 \\
 269 & 68886 & 69181 & 295 & 69094.33 & 208.33 \\
 270 & 69406 & 69700 & 294 & 69615.41 & 209.41 \\
 271 & 69927 & 70221 & 294 & 70138.47 & 211.47 \\
 272 & 70451 & 70744 & 293 & 70663.49 & 212.49 \\
 273 & 70977 & 71269 & 292 & 71190.49 & 213.49 \\
 274 & 71507 & 71796 & 289 & 71719.45 & 212.45 \\
 275 & 72038 & 72325 & 287 & 72250.38 & 212.38 \\
 276 & 72571 & 72856 & 285 & 72783.28 & 212.28 \\
 277 & 73105 & 73389 & 284 & 73318.14 & 213.14 \\
 278 & 73641 & 73924 & 283 & 73854.98 & 213.98 \\
 279 & 74179 & 74461 & 282 & 74393.78 & 214.78 \\
 280 & 74719 & 75000 & 281 & 74934.56 & 215.56 \\
 281 & 75261 & 75541 & 280 & 75477.30 & 216.30 \\
 282 & 75804 & 76084 & 280 & 76022.01 & 218.01 \\
 283 & 76349 & 76629 & 280 & 76568.69 & 219.69 \\
 284 & 76897 & 77176 & 279 & 77117.34 & 220.34 \\
 285 & 77446 & 77725 & 279 & 77667.95 & 221.95 \\
 286 & 77997 & 78276 & 279 & 78220.54 & 223.54 \\
 287 & 78550 & 78829 & 279 & 78775.10 & 225.10 \\
 288 & 79105 & 79384 & 279 & 79331.62 & 226.62 \\
 289 & 79662 & 79941 & 279 & 79890.12 & 228.12 \\ 
 290 & 80222 & 80500 & 278 & 80450.58 & 228.58 \\
 291 & 80785 & 81061 & 276 & 81013.01 & 228.01 \\
 292 & 81349 & 81624 & 275 & 81577.42 & 228.42 \\
 293 & 81915 & 82189 & 274 & 82143.79 & 228.79 \\
 294 & 82483 & 82756 & 273 & 82712.13 & 229.13 \\
 295 & 83052 & 83325 & 273 & 83282.44 & 230.44 \\
 296 & 83624 & 83896 & 272 & 83854.72 & 230.72 \\
 297 & 84197 & 84469 & 272 & 84428.97 & 231.97 \\
 298 & 84772 & 85044 & 272 & 85005.19 & 233.19 \\
 299 & 85349 & 85621 & 272 & 85583.38 & 234.38 \\
 300 & 85928 & 86200 & 272 & 86163.54 & 235.54 \\
 301 & 86509 & 86781 & 272 & 86745.67 & 236.67 \\
 \end{tabular}
}\end{minipage}\hspace{-1mm}
\begin{minipage}{0.5\textwidth}\centering
{\scriptsize
\begin{tabular}{cccccc}
$n$ & $|W(n)|$ & $h(n)$ & $e_1$ & $\lastFunc(n)$ & $e_2$\\
\hline
 302 & 87091 & 87419 & 328 & 87329.77 & 238.77 \\
 303 & 87676 & 88003 & 327 & 87915.84 & 239.84 \\
 304 & 88262 & 88589 & 327 & 88503.88 & 241.88 \\
 305 & 88850 & 89177 & 327 & 89093.90 & 243.90 \\
 306 & 89441 & 89767 & 326 & 89685.88 & 244.88 \\
 307 & 90034 & 90359 & 325 & 90279.83 & 245.83 \\
 308 & 90631 & 90953 & 322 & 90875.75 & 244.75 \\
 309 & 91229 & 91549 & 320 & 91473.64 & 244.64 \\
 310 & 91829 & 92147 & 318 & 92073.50 & 244.50 \\
 311 & 92430 & 92747 & 317 & 92675.33 & 245.33 \\
 312 & 93033 & 93349 & 316 & 93279.13 & 246.13 \\
 313 & 93638 & 93953 & 315 & 93884.91 & 246.91 \\
 314 & 94245 & 94559 & 314 & 94492.65 & 247.65 \\
 315 & 94854 & 95167 & 313 & 95102.37 & 248.37 \\
 316 & 95464 & 95777 & 313 & 95714.05 & 250.05 \\
 317 & 96076 & 96389 & 313 & 96327.71 & 251.71 \\
 318 & 96691 & 97003 & 312 & 96943.33 & 252.33 \\
 319 & 97307 & 97619 & 312 & 97560.93 & 253.93 \\
 320 & 97925 & 98237 & 312 & 98180.50 & 255.50 \\
 321 & 98545 & 98857 & 312 & 98802.04 & 257.04 \\
 322 & 99167 & 99479 & 312 & 99425.55 & 258.55 \\
 323 & 99791 & 100103 & 312 & 100051.03 & 260.03 \\
 324 & 100417 & 100729 & 312 & 100678.48 & 261.48 \\
 325 & 101046 & 101357 & 311 & 101307.90 & 261.90 \\
 326 & 101678 & 101987 & 309 & 101939.29 & 261.29 \\
 327 & 102311 & 102619 & 308 & 102572.66 & 261.66 \\
 328 & 102946 & 103253 & 307 & 103208.00 & 262.00 \\
 329 & 103583 & 103889 & 306 & 103845.30 & 262.30 \\
 330 & 104221 & 104527 & 306 & 104484.58 & 263.58 \\
 331 & 104862 & 105167 & 305 & 105125.83 & 263.83 \\
 332 & 105504 & 105809 & 305 & 105769.05 & 265.05 \\
 333 & 106148 & 106453 & 305 & 106414.25 & 266.25 \\
 334 & 106794 & 107099 & 305 & 107061.41 & 267.41 \\
 335 & 107442 & 107747 & 305 & 107710.55 & 268.55 \\
 336 & 108092 & 108397 & 305 & 108361.66 & 269.66 \\
 337 & 108743 & 109106 & 363 & 109014.74 & 271.74 \\
 338 & 109397 & 109759 & 362 & 109669.79 & 272.79 \\
 339 & 110052 & 110414 & 362 & 110326.81 & 274.81 \\
 340 & 110709 & 111071 & 362 & 110985.80 & 276.80 \\
 341 & 111368 & 111730 & 362 & 111646.77 & 278.77 \\
 342 & 112030 & 112391 & 361 & 112309.71 & 279.71 \\
 343 & 112694 & 113054 & 360 & 112974.62 & 280.62 \\
 344 & 113362 & 113719 & 357 & 113641.50 & 279.50 \\
 345 & 114031 & 114386 & 355 & 114310.35 & 279.35 \\
 346 & 114702 & 115055 & 353 & 114981.18 & 279.18 \\
 347 & 115374 & 115726 & 352 & 115653.98 & 279.98 \\
 348 & 116048 & 116399 & 351 & 116328.74 & 280.74 \\
 349 & 116724 & 117074 & 350 & 117005.49 & 281.49 \\
 350 & 117402 & 117751 & 349 & 117684.20 & 282.20 \\
 351 & 118082 & 118430 & 348 & 118364.89 & 282.89 \\
 352 & 118763 & 119111 & 348 & 119047.55 & 284.55 \\
 353 & 119446 & 119794 & 348 & 119732.18 & 286.18 \\
 354 & 120132 & 120479 & 347 & 120418.78 & 286.78 \\
 355 & 120819 & 121166 & 347 & 121107.35 & 288.35 \\
 356 & 121508 & 121855 & 347 & 121797.90 & 289.90 \\
 357 & 122199 & 122546 & 347 & 122490.42 & 291.42 \\
 358 & 122892 & 123239 & 347 & 123184.91 & 292.91 \\
 359 & 123587 & 123934 & 347 & 123881.38 & 294.38 \\
 360 & 124284 & 124631 & 347 & 124579.82 & 295.82 \\
 361 & 124983 & 125330 & 347 & 125280.23 & 297.23 \\
 362 & 125685 & 126031 & 346 & 125982.61 & 297.61 \\
 363 & 126390 & 126734 & 344 & 126686.97 & 296.97 \\
 364 & 127096 & 127439 & 343 & 127393.29 & 297.29 \\
 365 & 127804 & 128146 & 342 & 128101.60 & 297.60 \\
 366 & 128514 & 128855 & 341 & 128811.87 & 297.87 \\
\end{tabular}
}\end{minipage}
\caption{Values of $|W(n)|$, $h(n)$, $\lastFunc(n)$, and errors for $237 \le n \le 366$}
\label{Size W(n) table2}
\end{table}

\newpage

\begin{table}[H]
\begin{minipage}{0.5\linewidth}\centering
{\scriptsize
\begin{tabular}{cccccc}
$n$ & $|W(n)|$ & $h(n)$ & $e_1$ & $\lastFunc(n)$ & $e_2$\\
\hline
 367 & 129225 & 129566 & 341 & 129524.12 & 299.12 \\
 368 & 129939 & 130279 & 340 & 130238.34 & 299.34 \\
 369 & 130654 & 130994 & 340 & 130954.53 & 300.53 \\
 370 & 131371 & 131711 & 340 & 131672.69 & 301.69 \\
 371 & 132090 & 132430 & 340 & 132392.83 & 302.83 \\
 372 & 132811 & 133151 & 340 & 133114.94 & 303.94 \\
 373 & 133534 & 133874 & 340 & 133839.03 & 305.03 \\
 374 & 134258 & 134659 & 401 & 134565.09 & 307.09 \\
 375 & 134985 & 135385 & 400 & 135293.12 & 308.12 \\
 376 & 135713 & 136113 & 400 & 136023.12 & 310.12 \\
 377 & 136443 & 136843 & 400 & 136755.10 & 312.10 \\
 378 & 137175 & 137575 & 400 & 137489.05 & 314.05 \\
 379 & 137910 & 138309 & 399 & 138224.97 & 314.97 \\
 380 & 138647 & 139045 & 398 & 138962.87 & 315.87 \\
 381 & 139386 & 139783 & 397 & 139702.74 & 316.74 \\
 382 & 140129 & 140523 & 394 & 140444.58 & 315.58 \\
 383 & 140873 & 141265 & 392 & 141188.40 & 315.40 \\
 384 & 141619 & 142009 & 390 & 141934.19 & 315.19 \\
 385 & 142366 & 142755 & 389 & 142681.95 & 315.95 \\
 386 & 143115 & 143503 & 388 & 143431.69 & 316.69 \\
 387 & 143866 & 144253 & 387 & 144183.40 & 317.40 \\
 388 & 144619 & 145005 & 386 & 144937.08 & 318.08 \\
 389 & 145374 & 145759 & 385 & 145692.74 & 318.74 \\
 390 & 146130 & 146515 & 385 & 146450.37 & 320.37 \\
 391 & 146888 & 147273 & 385 & 147209.98 & 321.98 \\
 392 & 147649 & 148033 & 384 & 147971.56 & 322.56 \\
 393 & 148411 & 148795 & 384 & 148735.11 & 324.11 \\
 394 & 149175 & 149559 & 384 & 149500.64 & 325.64 \\
 395 & 149941 & 150325 & 384 & 150268.14 & 327.14 \\
 396 & 150709 & 151093 & 384 & 151037.61 & 328.61 \\
 397 & 151479 & 151863 & 384 & 151809.06 & 330.06 \\
 398 & 152251 & 152635 & 384 & 152582.48 & 331.48 \\
 399 & 153025 & 153409 & 384 & 153357.88 & 332.88 \\
 400 & 153801 & 154185 & 384 & 154135.25 & 334.25 \\
 401 & 154580 & 154963 & 383 & 154914.59 & 334.59 \\
 402 & 155362 & 155743 & 381 & 155695.91 & 333.91 \\
 403 & 156145 & 156525 & 380 & 156479.20 & 334.20 \\
 404 & 156930 & 157309 & 379 & 157264.47 & 334.47 \\
 405 & 157717 & 158095 & 378 & 158051.71 & 334.71 \\
 406 & 158505 & 158883 & 378 & 158840.92 & 335.92 \\
 407 & 159296 & 159673 & 377 & 159632.11 & 336.11 \\
 408 & 160088 & 160465 & 377 & 160425.27 & 337.27 \\
 409 & 160882 & 161259 & 377 & 161220.41 & 338.41 \\
 410 & 161678 & 162055 & 377 & 162017.52 & 339.52 \\
 411 & 162476 & 162853 & 377 & 162816.61 & 340.61 \\
 412 & 163276 & 163653 & 377 & 163617.67 & 341.67 \\
 413 & 164077 & 164517 & 440 & 164420.70 & 343.70 \\
 414 & 164881 & 165320 & 439 & 165225.71 & 344.71 \\
 415 & 165686 & 166125 & 439 & 166032.70 & 346.70 \\
 416 & 166493 & 166932 & 439 & 166841.66 & 348.66 \\
 417 & 167302 & 167741 & 439 & 167652.59 & 350.59 \\
 418 & 168114 & 168552 & 438 & 168465.50 & 351.50 \\
 419 & 168927 & 169365 & 438 & 169280.38 & 353.38 \\
 420 & 169743 & 170180 & 437 & 170097.23 & 354.23 \\
 421 & 170561 & 170997 & 436 & 170916.07 & 355.07 \\
 422 & 171383 & 171816 & 433 & 171736.87 & 353.87 \\
 423 & 172206 & 172637 & 431 & 172559.65 & 353.65 \\
 424 & 173031 & 173460 & 429 & 173384.41 & 353.41 \\
 425 & 173857 & 174285 & 428 & 174211.14 & 354.14 \\
 426 & 174685 & 175112 & 427 & 175039.84 & 354.84 \\
 427 & 175515 & 175941 & 426 & 175870.52 & 355.52 \\
 428 & 176347 & 176772 & 425 & 176703.18 & 356.18 \\
 429 & 177181 & 177605 & 424 & 177537.81 & 356.81 \\
 430 & 178016 & 178440 & 424 & 178374.41 & 358.41 \\
 431 & 178853 & 179277 & 424 & 179212.99 & 359.99 \\
\end{tabular}
}\end{minipage}\hspace{-1mm}
\begin{minipage}{0.5\textwidth}\centering
{\scriptsize
\begin{tabular}{cccccc}
$n$ & $|W(n)|$ & $h(n)$ & $e_1$ & $\lastFunc(n)$ & $e_2$\\
\hline
 432 & 179693 & 180116 & 423 & 180053.55 & 360.55 \\
 433 & 180534 & 180957 & 423 & 180896.07 & 362.07 \\
 434 & 181377 & 181800 & 423 & 181740.58 & 363.58 \\
 435 & 182222 & 182645 & 423 & 182587.06 & 365.06 \\
 436 & 183069 & 183492 & 423 & 183435.51 & 366.51 \\
 437 & 183918 & 184341 & 423 & 184285.94 & 367.94 \\
 438 & 184769 & 185192 & 423 & 185138.35 & 369.35 \\
 439 & 185622 & 186045 & 423 & 185992.73 & 370.73 \\
 440 & 186477 & 186900 & 423 & 186849.08 & 372.08 \\
 441 & 187334 & 187757 & 423 & 187707.42 & 373.42 \\
 442 & 188194 & 188616 & 422 & 188567.72 & 373.72 \\
 443 & 189057 & 189477 & 420 & 189430.00 & 373.00 \\
 444 & 189921 & 190340 & 419 & 190294.26 & 373.26 \\
 445 & 190787 & 191205 & 418 & 191160.49 & 373.49 \\
 446 & 191655 & 192072 & 417 & 192028.70 & 373.70 \\
 447 & 192524 & 192941 & 417 & 192898.88 & 374.88 \\
 448 & 193396 & 193812 & 416 & 193771.04 & 375.04 \\
 449 & 194269 & 194685 & 416 & 194645.17 & 376.17 \\
 450 & 195144 & 195560 & 416 & 195521.28 & 377.28 \\
 451 & 196021 & 196437 & 416 & 196399.37 & 378.37 \\
 452 & 196900 & 197316 & 416 & 197279.43 & 379.43 \\
 453 & 197781 & 198197 & 416 & 198161.46 & 380.46 \\
 454 & 198663 & 199144 & 481 & 199045.48 & 382.48 \\
 455 & 199548 & 200028 & 480 & 199931.46 & 383.46 \\
 456 & 200434 & 200914 & 480 & 200819.43 & 385.43 \\
 457 & 201322 & 201802 & 480 & 201709.36 & 387.36 \\
 458 & 202212 & 202692 & 480 & 202601.28 & 389.28 \\
 459 & 203105 & 203584 & 479 & 203495.17 & 390.17 \\
 460 & 203999 & 204478 & 479 & 204391.04 & 392.04 \\
 461 & 204895 & 205374 & 479 & 205288.88 & 393.88 \\
 462 & 205794 & 206272 & 478 & 206188.69 & 394.69 \\
 463 & 206695 & 207172 & 477 & 207090.49 & 395.49 \\
 464 & 207600 & 208074 & 474 & 207994.26 & 394.26 \\
 465 & 208506 & 208978 & 472 & 208900.00 & 394.00 \\
 466 & 209414 & 209884 & 470 & 209807.72 & 393.72 \\
 467 & 210323 & 210792 & 469 & 210717.42 & 394.42 \\
 468 & 211234 & 211702 & 468 & 211629.09 & 395.09 \\
 469 & 212147 & 212614 & 467 & 212542.74 & 395.74 \\
 470 & 213062 & 213528 & 466 & 213458.37 & 396.37 \\
 471 & 213979 & 214444 & 465 & 214375.97 & 396.97 \\
 472 & 214897 & 215362 & 465 & 215295.55 & 398.55 \\
 473 & 215817 & 216282 & 465 & 216217.10 & 400.10 \\
 474 & 216740 & 217204 & 464 & 217140.63 & 400.63 \\
 475 & 217664 & 218128 & 464 & 218066.14 & 402.14 \\
 476 & 218590 & 219054 & 464 & 218993.62 & 403.62 \\
 477 & 219518 & 219982 & 464 & 219923.08 & 405.08 \\
 478 & 220448 & 220912 & 464 & 220854.51 & 406.51 \\
 479 & 221380 & 221844 & 464 & 221787.92 & 407.92 \\
 480 & 222314 & 222778 & 464 & 222723.31 & 409.31 \\
 481 & 223249 & 223779 & 530 & 223660.67 & 411.67 \\
 482 & 224187 & 224716 & 529 & 224600.01 & 413.01 \\
 483 & 225127 & 225655 & 528 & 225541.33 & 414.33 \\
 484 & 226069 & 226596 & 527 & 226484.62 & 415.62 \\
 485 & 227014 & 227539 & 525 & 227429.89 & 415.89 \\
 486 & 227962 & 228484 & 522 & 228377.13 & 415.13 \\
 487 & 228911 & 229431 & 520 & 229326.36 & 415.36 \\
 488 & 229862 & 230380 & 518 & 230277.55 & 415.55 \\
 489 & 230815 & 231331 & 516 & 231230.73 & 415.73 \\
 490 & 231769 & 232284 & 515 & 232185.88 & 416.88 \\
 491 & 232726 & 233239 & 513 & 233143.01 & 417.01 \\
 492 & 233684 & 234196 & 512 & 234102.11 & 418.11 \\
 493 & 234644 & 235155 & 511 & 235063.19 & 419.19 \\
 494 & 235606 & 236116 & 510 & 236026.25 & 420.25 \\
 495 & 236570 & 237079 & 509 & 236991.29 & 421.29 \\
 496 & 237536 & 238044 & 508 & 237958.30 & 422.30 \\
\end{tabular}
}\end{minipage}
\caption{Values of $|W(n)|$, $h(n)$, $\lastFunc(n)$, and errors for $367 \le n \le 496$}
\label{Size W(n) table3}
\end{table}

\newpage

\begin{table}[H]
\begin{minipage}{0.5\linewidth}\centering
{\scriptsize
\begin{tabular}{cccccc}
$n$ & $|W(n)|$ & $h(n)$ & $e_1$ & $\lastFunc(n)$ & $e_2$\\
\hline
 497 & 238503 & 239011 & 508 & 238927.29 & 424.29 \\
 498 & 239473 & 239980 & 507 & 239898.25 & 425.25 \\
 499 & 240444 & 240951 & 507 & 240871.19 & 427.19 \\
 500 & 241417 & 241924 & 507 & 241846.11 & 429.11 \\
 501 & 242392 & 242899 & 507 & 242823.01 & 431.01 \\
 502 & 243370 & 243876 & 506 & 243801.88 & 431.88 \\
 503 & 244349 & 244855 & 506 & 244782.73 & 433.73 \\
\end{tabular}
}\end{minipage}\hspace{-1mm}
\begin{minipage}{0.5\textwidth}\centering
{\scriptsize
\begin{tabular}{cccccc}
$n$ & $|W(n)|$ & $h(n)$ & $e_1$ & $\lastFunc(n)$ & $e_2$\\
\hline
 504 & 245330 & 245836 & 506 & 245765.55 & 435.55 \\
 505 & 246313 & 246819 & 506 & 246750.35 & 437.35 \\
 506 & 247299 & 247804 & 505 & 247737.13 & 438.13 \\  
 507 & 248287 & 248791 & 504 & 248725.89 & 438.89 \\
 508 & 249279 & 249780 & 501 & 249716.62 & 437.62 \\
 509 & 250272 & 250771 & 499 & 250709.33 & 437.33 \\
 510 & 251267 & 251764 & 497 & 251704.02 & 437.02 \\
\end{tabular}
}\end{minipage}
\caption{Values of $|W(n)|$, $h(n)$, $\lastFunc(n)$, and errors for $497 \le n \le 510$}
\label{Size W(n) table4}
\end{table} 


\end{document}